\documentclass{article}

\usepackage[english]{babel}
\usepackage[utf8]{inputenc}
\usepackage[T1]{fontenc}
\usepackage{lipsum}

\usepackage[left=3cm,right=3cm,top=3cm,bottom=3cm]{geometry}
\usepackage{indentfirst}
\usepackage{enumitem}

\usepackage{graphicx}
\usepackage{xcolor}
\usepackage{caption}

\usepackage{float}
\usepackage{tikz,tkz-tab}
\usepackage{booktabs}
\usepackage{array}

\usepackage{amssymb}
\usepackage{mathtools}
\usepackage{stmaryrd}
\usepackage{faktor}
\usepackage{bm}
\usepackage{tikz-cd}

\usepackage{hyperref}
\usepackage[amsmath,thref,amsthm,hyperref,thmmarks]{ntheorem}
\usepackage{cases}

 \newtheorem{thm}{Theorem}
 \newtheorem{definition}[thm]{Definition}
 \newtheorem{prop}[thm]{Proposition}
 \newtheorem{lemme}[thm]{Lemma}
 
 \newtheorem{cor}[thm]{Corollary}

 \theoremstyle{remark}

\theoremstyle{remark}
\newtheorem{rem}[thm]{Remark}

\newcommand{\delbar}{{\bar\partial}}
\newcommand{\del}{\partial}

\newcommand{\grad}{\text{grad}_{g_\varepsilon}}
\newcommand{\gradV}{\text{grad}_{g}}

\newcommand{\C}{\mathbb{C}}
\newcommand{\Z}{\mathbb{Z}}
\newcommand{\R}{\mathbb{R}}
\newcommand{\N}{\mathbb{N}}

\renewcommand{\L}{\mathcal{L}}
\newcommand{\Lich}{\mathbb{L}}

\newcommand{\CZ}{\C^2/\Z_2}
\newcommand{\AC}{\mathcal{AC}}
\newcommand{\Ham}{\text{Ham}}

\newcommand{\Ric}{\text{Ric}}

\newcommand{\vol}{\text{vol}}

\title{Almost-Kähler smoothings of compact complex surfaces with $A_1$ singularities.}
\author{Caroline Vernier}
\date{\today}

\begin{document}

\maketitle

\begin{abstract}
This paper is concerned with the existence of metrics of constant Hermitian scalar curvature on almost-Kähler manifolds obtained as smoothings of a constant scalar curvature Kähler orbifold, with $A_1$ singularities. More precisely, given such an orbifold that does not admit nontrivial holomorphic vector fields, we show that an almost-Kähler smoothing $(M_\varepsilon, \omega_\varepsilon)$ admits an almost-Kähler structure $(\hat J_\varepsilon, \hat g_\varepsilon)$ of constant Hermitian curvature. Moreover, we show that for $\varepsilon>0$ small enough, the $(M_\varepsilon, \omega_\varepsilon)$ are all symplectically equivalent to a fixed symplectic manifold $(\hat M, \hat \omega)$ in which there is a surface $S$ homologous to a 2-sphere, such that $[S]$ is a vanishing cycle that admits a representant that is Hamiltonian stationary for $\hat g_\varepsilon$.
\end{abstract}

\section{Introduction}

\subsection{Context: gluing methods in Kähler geometry.}

Let $M$ be a compact complex manifold of Kähler type. The program of Calabi is concerned with the existence of canonical metrics in a given Kähler class $\Omega$ on $M$. More specifically, Calabi proposed the study of the functional
\begin{equation*}
	\omega \in \Omega_{>0} \mapsto \int_M s(\omega)^2\, \frac{\omega^m}{m!};
\end{equation*}
here $\Omega_{>0}$ denotes the set of definite positive representants of the cohomology class $\Omega$, and $s(\omega)$ is the scalar curvature of the associated metric. The critical points of this functional are called extremal metrics, and they are the candidates for canonical metrics in this framework. 

Computing the corresponding Euler-Lagrange equation, one obtains that a Kähler metric is extremal if and only if the Hamiltonian vector field $X_{s(\omega)}$ is real holomorphic. In particular, constant scalar curvature metrics are extremal, and both notions coincide if $M$ admits no non-trivial holomorphic vector field. 

\medskip
 Non-trivial holomorphic vector fields appear as an obstruction in constructions of constant scalar curvature metrics. More precisely, on a Kähler manifold $(M, J, \omega)$, the obstructions on the structure of the Lie algebra $\mathfrak h(M,J)$ of holomorphic vector fields found by Matsushima \cite{Mat}, or the Futaki invariant \cite{Fut}, involve the following subset of $\mathfrak h(M,J)$:
 \begin{equation*}
 	\mathfrak{h}_0(M,J) = \{X \in \mathfrak h(M,J), \exists p\in M\ |\ X(p) = 0\}.
 \end{equation*}
 On a Kähler manifold (or orbifold), $\mathfrak{h}_0(M)$ form a Lie subalgebra of the Lie algebra $\mathfrak h(M,J)$ (see for instance \cite{LebSim}, Theorem 1). Therefore, it will be natural to assume that $\mathfrak h_0(M,J) = \{0\}$, to ensure that said obstructions do not appear.

This will be the case if the group of automorphisms of the $(M,J)$ is discrete. However, it is not a necessary condition; if $M$ is a torus, obtained as the quotient of $\C^2$ by a lattice, we do have $\mathfrak{h}_0(M) = \{0\}$, as it turns out in this case that all holomorphic vector fields are parallel.

\medskip

The existence of canonical metrics on a given Kähler manifold is an open problem in general. As a consequence, the construction of classes of examples through gluing methods has been the focus of many works. For instance, Arezzo and Pacard \cite{ArePac,ArePac2} have obtained constant scalar curvature Kähler (cscK) metrics on blow-up of cscK manifolds or orbifolds; Arezzo, Lena and Mazzieri have generalized these methods to resolutions of compact orbifolds with isolated singularities; Biquard and Rollin \cite{BiqRol} have studied smoothings of canonical singularities, generalizing results by Spotti \cite{Spo} on smoothings of $A_1$ singularities in the Kähler-Einstein case. In the case of extremal metrics, one may cite the works of Arezzo, Pacard and Singer \cite{ArePacSin} or Szekelyhidi \cite{Sze3,Sze4}.

Another aspect of the existence problem for extremal metrics is its generalisation to almost-Kähler manifolds. These are symplectic manifolds $(M, \omega)$ endowed with a compatible almost-complex structure, that is not assumed to be integrable. The space $\AC_\omega$ of almost complex structures is known to be a contractible Fréchet space, endowed with a natural Kähler structure. The action of the group of Hamiltonian symplectomorphisms acts on $\AC_\omega$ by pullback. The key observation, due to Donaldson \cite{Don3} (generalizing Fujiki's work \cite{Fuj} to the non-integrable setting), is that this action is Hamiltonian, with moment map given by the Hermitian scalar curvature of $(M, \omega, J)$, which is to say the trace of the curvature of the Chern connection on the anticanonical bundle.

Thus, the suitable reframing of the problem is then the study of the functional
\begin{equation*}
	J \in \AC_\omega \mapsto \int_M (s^\nabla(J))^2\, \frac{\omega^m}{m!},
\end{equation*}
which coincide with the Calabi functional in the Kähler case.
In this direction, Lejmi \cite{Lej1} has generalised many notions linked to the existence problem of canonical metrics, and its relation to K-stability, such as the Futaki invariant. 
In another direction, Weinkove et al. \cite{Wei,ChuTosWei} study the Calabi-Yau equation on an almost-Kähler 4-manifold $(M, \omega, J)$.

\subsection{Statement of results.}

Let $(M^4, \omega_M, J_M)$ be a compact Kähler orbifold with isolated singularities of type $A_1$, denoted $p_1\dots,,p_\ell$. This means that $M$ is endowed with a holomorphic atlas that maps neighborhoods of the $p_i$ to neighborhoods of 0 in $\CZ$.

Such orbifold surfaces, and more generally surfaces with canonical singularities, arise naturally by global quotient constructions, as well as in the context of pluricanonical Kodaira `embeddings' of surfaces of general type. Such maps are obtained by contraction of divisors of self-intersection -2 in a surface of general type, which results in canonical singularities.

\medskip

In section \ref{sec:Darboux}, we detail the construction of a family of smooth symplectic manifolds $M_\varepsilon$ indexed by a parameter $\varepsilon \in (0, \varepsilon_0)$, called a family of \emph{smoothings} of the orbifold $(M, \omega_M)$. We will obtain these smoothings by a symplectic connected sum between $M$ and an ALE Kähler model $(X \simeq T^*S^2, J_X, \omega_X)$, Ricci flat, and with exact symplectic form $\omega_X$. The construction of this ALE metric is detailed in the Annex.

For now, we simply highlight the fundamental properties of the smoothing.
\begin{enumerate}
	\item The manifold $M_\varepsilon$ will split into $M_\varepsilon = (M \setminus \cup_i B(p_i, r(\varepsilon))) \cup K_\varepsilon$, where $K_\varepsilon$ is diffeomorphic to a compact neighborhoods $\tilde K_\varepsilon$ of the zero section in $T^*S^2$. Moreover $r(\varepsilon)$ goes to 0 as $\varepsilon$ goes to 0, and $T^*S^2 = \cup_\varepsilon \tilde K_\varepsilon$. 
	\item $M_\varepsilon$ is endowed with a symplectic form $\omega_\varepsilon$ such that, on the one hand, the injection $(M \setminus \cup_i B(p_i, r(\varepsilon)) \hookrightarrow M_\varepsilon$, sends $\omega_M$ to $\omega_\varepsilon$, and, on the other hand, the diffeomorphism $\psi_\varepsilon : K_\varepsilon \rightarrow \tilde K_\varepsilon$ sends $\varepsilon^{-2}\omega_\varepsilon$ to $\omega_X$.
\end{enumerate}

From these properties, we will see in Lemma \ref{lemme:cohomIndEps} that the manifolds $M_\varepsilon$ are all diffeomorphic, and actually symplectomorphic.
Indeed, there is a canonical injection
\begin{equation}
	H^2_c(M\setminus \{p_1, \dots, p_\ell\}) \hookrightarrow H^2(M_\varepsilon, \R)
\end{equation}
that sends $[\omega_M]$ to $[\omega_\varepsilon]$. In this sense, the cohomology classes of $[\omega_\varepsilon]$ all agree.

\medskip

Furthermore, the identifications of regions of $M_\varepsilon$ with regions of $M$ and $X$ enable us to make sense of the convergence, when $\varepsilon$ goes to zero, of sequences of functions (or tensors) $f_\varepsilon : M_\varepsilon \rightarrow \R$  on compact sets of $M^* := M \setminus \{p_1, \dots, p_\ell\}$ on the one hand, and on compact sets of $X$ on the other hand.

Making this construction precise is the object of section \ref{sec:Darboux}. In this situation, we obtain the following result.
\begin{thm}\label{thm:mainresult}
   Assume that $(M,J)$ admits no nontrivial holomorphic vector fields that vanish somewhere on $M$, and that $(M,\omega_M, J_M)$ is Kähler, of constant scalar curvature. For a positive parameter $\varepsilon$ small enough, we endow the symplectic manifolds $(M_\varepsilon, \omega_\varepsilon)$ with a family of smooth compatible almost-Kähler structures $J_\varepsilon,  g_\varepsilon$ of constant Hermitian scalar curvature, such that, when $\varepsilon$ goes to zero,
   \begin{itemize}
   	\item  The sequence of almost complex structures $ J_\varepsilon$ converges, in  $\mathcal C^{k,\alpha}$-norm, to the orbifold complex structure $J_M$, on every compact set of $M^*$, for every $k \in \N$.
   	\item The pushed-forward almost complex structures  $ (\psi_\varepsilon)_* J_\varepsilon$ converges, in any $\mathcal C^{k,\alpha}$-norm, to the ALE complex structure $J_X$, on every compact set of $X$, for every $k\in \N$.
   \end{itemize}
\end{thm}

\begin{rem}
In \cite{BiqRol}, the same result is obtained in the case where $J$ in integrable. 

However, the methods presented here are new. In usual gluing methods, the deformation of the approximate solution into a canonical metrics is obtained by adding a potential function. The $\del\delbar$-lemma makes such an approach natural in the Kähler setting. 
As we will see, this approach does not work so well in the almost-Kähler setting. In dimension 4, `almost-Kähler potentials' have been used by Weinkove \cite{Wei} in his study of the Calabi-Yau equation on almost-Kähler manifolds, and by Lejmi \cite{Lej3}. However, this method involves the use of pseudo-differential operators.

To prove our result, we will instead turn to an approach inspired by Fujiki \cite{Fuj} and Donaldson's \cite{Don3} moment map picture for canonical metrics.

\medskip

Besides the almost-Kähler setting, an element of novelty here is that the cohomology class of the $\omega_\varepsilon$ is different from the one obtained with gluing techniques like Arezzo and Pacard's.
On blow-ups, constant curvature metric are usually obtained in a class of the form
\begin{equation*}
  \Omega = [\omega] - \sum_i \varepsilon^2 \lambda_i [E_i],
\end{equation*}
where the $[E_i]$ are Poincaré-dual to the holomorphic exceptional divisor, and the $\lambda_i$ are positive coefficients. Instead, in our construction, the zero section of $T^*S^2$ is included in the compact sets $\tilde K_\varepsilon$, thus, via the identification $\tilde K_\varepsilon \rightarrow K_\varepsilon \subset M_\varepsilon$, yields a Lagrangian sphere $S_\varepsilon$:
\begin{equation*}
  [ \omega_\varepsilon]\cdot [S_\varepsilon] = 0.
\end{equation*}
\end{rem}

\medskip

This last observation enables us to extend another part of the results obtained by Biquard and Rollin in \cite{BiqRol}, namely the existence of a family of Hamiltonian stationary spheres corresponding to our family of metrics $ \hat g_\varepsilon$. 
Let $(M,\omega, J, g)$ be a Kähler (or almost-Kähler) manifold. A Hamiltonian stationary surface is a Lagrangian surface $L$ which is a critical point of the area functional under Hamiltonian deformations, which is to say that, for any smooth function $F \in \mathcal C^\infty(M)$, we have
\begin{equation}\label{eq:defHamStat}
	\frac{d}{ds}_{|s=0} Vol_g(\exp(sX_F)(L)) =0,
\end{equation}
where $\exp(sX_F)$ denotes the flow of the Hamiltonian vector field $X_F$. 
Such surfaces have been introduced and studied by Oh in \cite{Oh1,Oh2}; new examples generalizing Oh's have been obtained by Joyce, Lee and Schoen in \cite{JoyLeeScho}. Schoen and Wolfson \cite{SchoWolf} have studied  the existence of Lagrangian surfaces that minimize the area.

In this direction, we obtain:

\begin{thm}
  On $(M_\varepsilon, \omega_\varepsilon, J_\varepsilon)$, for $\varepsilon$ small enough, the Lagrangian sphere $S_\varepsilon$ admits a Hamiltonian deformation that is a Hamiltonian stationary 2-sphere for the metric $g_\varepsilon$.
\end{thm}

\subsection{Outline of the method.}

Let us now flesh out some details of the gluing construction. Following the gluing methods introduced by Arezzo and Pacard in \cite{ArePac}, we seek to endow a smooth manifold $M_\varepsilon$, obtained from $M$ by a connected sum construction with a suitable asymptotically locally euclidean (ALE) model $X$, with a constant Hermitian curvature structure.
\medskip

For such a construction to work, the ALE surface $X$ needs to be asymptotic to $\C^2 / \Z_2$, in the sense that the Riemannian metric and complex structure on $X$ converge to the Euclidean ones $J_0$, $g_0$ on $\CZ$ fast enough. This ALE model will be provided by smoothings 
\begin{equation}\label{eq:smoothingsA1}
	C_\varepsilon=\{z \in \C^ 3, z_1^2 + z_2^2 + z_3^2 = \varepsilon^2\}
\end{equation}
of the quotient singularity $\CZ$, which we identify to the cone
\begin{equation*}
	C=\{z \in \C^ 3, z_1^2 + z_2^2 + z_3^2 = 0\}.
\end{equation*}
For $\varepsilon > 0$, these are diffeomorphic to $T^*S^2$, which is endowed with Eguchi-Hanson's Ricci-flat metric and a complex structure that is a deformation of the one obtained when blowing up the quotient singularity $\C^2 / \Z_2$. We refer to the Annex for more details about the ALE model.
\medskip

\begin{rem} The minimal resolution of the $A_1$ singularity is an hyperKähler manifold biholomorphic to $T^*\C P1$. Our choice here consists of taking a different complex structure in the hyperKähler family. This observation is the starting point of the construction of Hamiltonian stationary spheres later on. 
\end{rem}
\medskip

The next step is to glue together $M$ and $X$ in a generalized connected sum, that is a smooth, compact manifold: we replace a very small neighborhood of each singularity $p_i$ of $M$ by a suitably scaled-down `ball' of large radius in $X$. 
Performing this construction in Darboux charts, we ensure that the obtained smooth manifold $M_\varepsilon$ is naturally endowed with a symplectic form $\omega_\varepsilon$. 

\medskip

Then, we endow $M_\varepsilon$ with an almost-Kähler structure $(\omega_\varepsilon, \hat J_\varepsilon, \hat g_\varepsilon)$ by patching together the model structures on $M$ and $X$. This `patching' comes at the price of the integrability of the obtained almost-complex structure $\hat J_\varepsilon$. Then, we perturb this approximate solution into an almost-Kähler structure of constant Hermitian scalar curvature. This requires to depart from `usual' gluing methods.

 \medskip

Since we are not working on a Kähler manifold, the Ricci and scalar curvature stemming from the Riemannian metric $\hat g_\varepsilon:=\omega_\varepsilon(\cdot, \hat J_\varepsilon\cdot)$ do not retain the same pleasant properties they have on a Kähler manifold. As a consequence, we study the \emph{Hermitian scalar curvature} instead; this is motivated by the moment-map point of view of Donaldson \cite{Don3}.
\medskip

Observe, moreover, that we have no appropriate notion of Kähler potential to perturb the symplectic form. Indeed, as was observed by Delanoe \cite{Del}, symplectic forms of the form
\begin{equation*}
  \omega_f := \omega_\varepsilon + d\hat J_\varepsilon df
\end{equation*}
are not $J_\varepsilon$-invariant, thus do not provide an almost-Kähler structure on $M_\varepsilon$. 
Instead, we are going to fix the symplectic form $\omega_\varepsilon$ and modify the almost complex structure $\hat J_\varepsilon$ along directions orthogonal to the Hamiltonian action, in a way that preserves compatibility with $\omega_\varepsilon$.

\medskip
This method allows us to rewrite the condition of constant Hermitian curvature as an elliptic fourth order PDE on $M_\varepsilon$. To solve it, we resort to a fixed-point method in suitable functional Banach spaces. It turns out that the linearisation of our PDE rewrites as the sum of the Lichnerowiz operator on $M_\varepsilon$ and an error term. Up to proper estimates of this error term, we may thus use the nice properties of the Lichnerowicz operators on the model spaces, namely the orbifold $M$ and the ALE surface $X$, to study the linearisation. This last step allows us to find a unique solution through an analogue of the inverse function theorem. 
  
\medskip

 As far as Theorem 2 is concerned, the key observation is that the zero section $S$ of $T^*S^2$ is Lagrangian for the symplectic form $\omega_X$; moreover it corresponds to the (holomorphic) zero section of $T^*\C P^1$ for another choice of complex structure in the hyperKahler family; it is then a consequence of Wirtinger's inequality that $S$ is minimal for Eguchi-Hanson's metric, which coincides with Stenzel's metric as a Riemannian structure. 

 This property is preserved when constructing the approximate solution: we obtain a Lagrangian minimal 2-sphere in $M_\varepsilon$. The idea is then to perturb $S$ inside its homology class by Hamiltonian transformation, and to use the implicit function theorem to obtain Hamiltonian-stationary representants for the nearby metrics $\hat g_\varepsilon$ obtained through the gluing process.

\subsection{Examples and perspectives.}

Let us exhibit some classes of singular surfaces to which our construction may apply.

\medskip

As was pointed out to us by R. Dervan, this construction applies to surfaces with $A_1$ singularities and ample canonical class, since such surfaces have negative first Chern class and thus are guaranteed to have a Kähler-Einstein metric (see Aubin \cite{Aub2}, and Kobayashi \cite{Kob2} for surfaces of general type)  and no nontrivial holomorphic vector fields (see \cite{Kob}, Chapter III, Theorem 2.1).

\medskip

In this direction, Miranda, in \cite{Mir}, studies a special case of complex surfaces with ample canonical bundle, that admit no smoothing. Thus, we may apply our construction, and these examples are outside the framework of the smoothing theorem obtained by Biquard and Rollin \cite{BiqRol}.

\medskip

Similarly, Catanese, in \cite{Cat}, exhibits a criterion for algebraic varieties with finite automorphism group, under which they admit no smoothing. His theorem encompasses the previously obtained obstructed examples, and the surfaces satisfying to this criterion have rational double points as singularities, and so do all of their deformations.  

\medskip

Finally, looking at the assumptions of the main theorem, some questions arise naturally, that open some perspectives:
\begin{itemize}
	\item Could we extend this construction to a wider range of singularities, such as canonical singularities~?
	\item What if the base manifold $M$ admits nontrivial holomorphic vector fields ? For instance, could we obtain a result in the line of \cite{Sze1} in our context?
\end{itemize}

Another question that arises is that of higher dimensions. However, in this case, it has been proven by Hein, Radeasconu and Suvaina in \cite{HeinRasSuv} that an ALE model asymptotic to a singularity $\C^m / G$ has to be isomorphic to a deformation of a resolution of the quotient singularity  $\C^m /G$. However, by Schlessinger's rigidity theorem \cite{Sch}, such singularities are actually rigid; as a consequence, in complex dimension greater than 3, the only available ALE model, up to biholomorphism, is the resolution of the singularity.

However, the double point in $\C^m$, identified to the cone
\begin{equation*}
\mathcal{C} = \{z \in \C^m, \sum_{i=1}^m z_i^2 = 0\}
\end{equation*}
still admits smoothings
\begin{equation*}
\mathcal{S}_\varepsilon = \{z \in \C^m, \sum_{i=1}^m z_i^2 = \varepsilon\}
\end{equation*}
that can be identified to the cotangent of the sphere $T^*S^m$. Stenzel's construction \cite{Ste} endows such smoothings with an ALE Ricci-flat metric. We could thus consider a similar construction, where the base $M$ has such conical singularities. 

\medskip

\subsection{Organisation of the paper.} In Section \ref{sec:AKgeneralities}, we begin with recalling the general properties of almost-Kähler manifolds that are needed in the paper; we discuss especially the space of amost complex structures compatible with a given symplectic form, as well as the properties of the Hermitian scalar curvature. In Section \ref{sec:Darboux}, we show the existence of Darboux charts around singularities in $M$ on the one hand, and outside a compact in $X$ on the other hand, in which the gluing is performed. Section \ref{sec:cpxmodif} is devoted to the construction of a compatible almost complex structure on $M_\varepsilon$, as well as estimates on its Nijenhuis tensor. In Section \ref{sec:Equation}, we tackle the analysis of the equation we want to solve on $M_\varepsilon$. The idea is to reduce the problem to a fixed-point problem in suitable Banach spaces, in the spirit of the Inverse Function Theorem, and to compare the intervening operators to the well-understood models on $M$ and $X$. Finally, Section \ref{sec: HamStat} is concerned with the proof of Theorem 2.

\paragraph{Acknowledgements.} I would like to thank my advisors Yann Rollin and Gilles Carron for their invaluable help and support during the maturation of this paper. I would also like to thank the CIRGET for their kind welcome and the stimulating work environment; special thanks to Vestislav Apostolov, who made this visit possible.

\section{Almost-Kähler preliminaries.}\label{sec:AKgeneralities}

Our construction will lead us into the realm of almost-Kähler geometry on a symplectic manifold. 
For the sake of completeness, we introduce here all the notions and identities that will appear in the main construction. 

Let $(V, \omega)$ be a symplectic manifold. First, we describe the space of almost complex structures compatible with $\omega$ and how it relates to Kahler classes in Kahler geometry. 
Then, we discuss several notion of scalar curvature on the almost Kähler manifold $(V,\omega, J)$, and explain why the Hermitian scalar curvature is most suited to our purposes.

\subsection{Almost complex structures compatible with a symplectic form. } First we give some background on which (almost)-complex structures are compatible with a given symplectic form. Let $(V, \omega)$ be a symplectic manifold. We consider the set of all almost complex structures on $V$ compatible with~$\omega$:
\begin{equation*}
  \mathcal{AC}_{\omega} = \{J \text{ section of End($TV$)}, \text{ such that } J^2 = -Id,\text{ and }\ g_J := \omega(\cdot, J\cdot) \text{ is a Riemann metric} \}.
\end{equation*}
Its tangent space at a point $J \in \AC_\omega$ is then given by:
\begin{equation*}
  T_J \AC_\omega = \{ A \text{ section  of End}(TV)\text{ such that } AJ = - JA,\ \omega(A\cdot, \cdot)+\omega(\cdot, A\cdot) = 0 \}.
\end{equation*}
Let $\mathcal{G}_{\omega}$ be the space of sections of Aut($TV$) that preserve $\omega$,
  \begin{equation*}
     \mathcal{G}_{\omega} = \Gamma(\text{Aut}(TV,\omega)) = \{\gamma : V\rightarrow \text{Aut}(TV),\ \omega(\gamma X, \gamma Y) = \omega(X,Y)\}.
  \end{equation*}
It can be understood as an infinite-dimensional Lie group, whose Lie algebra is then :
\begin{equation*}
\mathcal{L}_{\omega}=\Gamma(\text{End}(TV, \omega)) = \{a : V \rightarrow \text{End}(TV),\ \omega(aX, Y) + \omega(X, aY) = 0\}.
\end{equation*}
Then we have the following proposition, relating any to a.c.s. compatible with $\omega$:
\begin{prop}\label{propJcompatible}
  The action of  $\mathcal{G}_{\omega}$ on $\mathcal{AC}_{\omega}$ by conjugation is transitive. In particular, given $J_1$ and $J_2$ in $\mathcal{AC}_\omega$, there is an $a \in \mathcal{L}_{\omega}$ such that
\begin{equation*}
  J_2 = \exp(a)J_1\exp(-a);
\end{equation*}
moreover, the section $A$ is unique if we assume it anticommutes with $J_1$ and $J_2$. \\

Conversely, for any $J \in \AC_\omega$, any tangent $\dot J \in T_J\AC_\omega$ can be written as the tangent vector to a curve of this form:
\begin{equation*}
  \dot J = \frac{d}{dt}\bigg|_{t=0} \exp(ta)J\exp(-ta),
\end{equation*}
where $a = -\frac12 J\dot J$. \\

\end{prop}

\begin{proof}
  Observe that $P=-J_1J_2$ is symmetric positive definite with respect to both associated metrics $g_1 = \omega(\cdot, J_1 \cdot)$ and $g_2 = \omega(\cdot, J_2 \cdot)$. Thus we may write it $P = B^2$ for a symmetric definite positive matrix $B$. Write $B = \exp(b)$ and observe that $b$ anticommutes to both $J_1$ and $J_2$ to conclude. 
\end{proof}

\subsection{Action of Hamiltonian vector fields on  \texorpdfstring{$\AC_\omega$}{ the set of compatible acs}.}\label{sec:hamaction}

In the original construction proposed by Arezzo and Pacard, the ``connected sum'' on which the operation takes place is a complex manifold in a natural way, and one looks for a canonical metric in a Kähler class naturally obtained when performing the gluing. \ \\

Here we will lose this property on the connected sum. However, we will see that we can still endow it with a natural (family of) symplectic 2-forms. As a consequence, it will be more natural to keep this symplectic form fixed and move the obtained almost complex structure in $\AC_\omega$. 

In this section we explain ow one might perform this operation on a symplectic manifold $(V, \omega)$, and how, in the integrable case, this relates to the more traditional use of the $dd^c$-lemma to move around in a given Kähler class. 
\ \\

Since the natural structure on $V$ is the symplectic form $\omega$, it makes sense to use Hamiltonian vector fields to move the other structures around.
Thus, to a smooth function $f$ on $V$, we associate the Hamiltonian vector field $X_f$ defined by
\begin{equation*}
  df = \omega(X_f\,\cdot\,,\,\cdot\,).
\end{equation*}
A Hamiltonian vector field $X_f$ induces a variation $a$ of complex structures via the Lie derivative:
\begin{equation*}
  a = \frac12 \mathcal{L}_{X_f}J.
\end{equation*}
This variation is compatible with $\omega$ in the following sense:
\begin{lemme}
  The variation of complex structure $a$ is in $\mathcal{L}_{\omega}$. Moreover, $a$ anticommutes to $J$.
\end{lemme}

\begin{proof}
  We must first check that $\omega(aX, Y) +\omega(X, aY) =0$. To do this, we use that, since $X_f$ is hamiltonian, it preserves $\omega$, i.e.
  \begin{equation*}
    \mathcal{L}_{X_f}\omega =0.
  \end{equation*}
  Thus, since $g(X,Y) = \omega(X, J Y)$, we have that
  \begin{equation*}
    \mathcal{L}_{X_f}g(X,Y) = \omega(X, \mathcal{L}_{X_f}J Y).
  \end{equation*}
  But $\mathcal{L}_{X_f}g$ is a symmetric tensor, thus
  \begin{align*}
    \mathcal{L}_{X_f}g(X,Y) &= \mathcal{L}_{X_f}g(Y,X)\\
    &= \omega(Y, \mathcal{L}_{X_f}J X)\\
    &= -\omega(\mathcal{L}_{X_f}J X,Y).
  \end{align*}
  As for anticommuting with $J$, we have that
  \begin{align*}
    2aJ X &= (\mathcal{L}_{X_f}J)J X\\
    &= -\mathcal{L}_{X_f}X -J \mathcal{L}_{X_f}(J X)\\
    &= -J(\mathcal{L}_{X_f}J)X
  \end{align*}
  for any $X$.
\end{proof}

Thus, from Proposition \ref{propJcompatible}, we see that for any $t$, the almost complex structure
\begin{equation*}
  J_t = \exp(-ta)J\exp(ta)
\end{equation*}
is in $\AC_\omega$.
To $f \in \mathcal{C}^\infty(V)$, we may therefore associate 
\begin{equation}\label{eq:defJf}
	J_f := J_1.
\end{equation}

\medskip
As an heuristical aparté, let us now briefly explain how this construction can be related to the $dd^c$-lemma in Kähler geometry.

 The Lie group $\text{Ham}(V, \omega)$  of exact symplectomorphisms on a symplectic manifold $(V,\omega)$\footnote{Ham$(V,\omega)$ can be understood as the set of symplectomorphisms which are time-one value of the flow of a time-dependent Hamiltonian vector field.} acts on $\AC_{\omega}$ by pullback, and this action preserves $\omega$. Through the Hamiltonian construction, we identify the Lie algebra of $\Ham(V,\omega)$ with the set $E_0$ of smooth functions on $V$ with zero integral.

With this identification, the infinitesimal action is
\begin{equation*}
  P: f\in E_0 \mapsto \mathcal{L}_{X_f}J \in T_J\AC_\omega.
\end{equation*}

Observe that if $J$, $J'$ are \emph{integrable} complex structures, such that $J' = \phi^*J$ for some diffeomorphism $\phi$, then the associated Riemannian metric is given by:
\begin{equation}\label{hamactionmetric}
  g(J', \omega) = \phi^*g(J, (\phi^{-1})^*\omega);
\end{equation}
so if $\phi\in \Ham(V,\omega)$, these two metrics are isometric and have the same scalar curvature. This construction does not help to find constant scalar curvatures.

However, we may consider the \emph{complexified action} instead. We may not be able to complexify the Lie group, but we can consider
the complexified Lie algebra of zero-mean smooth functions with values in $\C$. This yields a complexified infinitesimal action
\begin{equation*}
  P : E_0^\C =\{H\in \mathcal{C}^\infty(V,\C), \int_V H \omega^2 =0\} \rightarrow T_J\AC_\omega.
\end{equation*}
The resulting foliation can be understood as the orbits of $\Ham^\C(V,\omega)$.

The (infinitesimal) action of a purely imaginary $\sqrt{-1}f$ is then given by $JP(f) = J \mathcal{L}_{X_f}J = \mathcal{L}_{JX_f}J$.
Thanks to \eqref{hamactionmetric}, we see that, at the riemannian level, this amounts to fixing $J$ and flowing $\omega$ along $-JX_f$. The obtained variation is then
\begin{equation*}
  - \mathcal{L}_{JX_f}\omega = -d\iota_{JX_f}\omega = 2i\del\delbar f.
\end{equation*}
so this construction is equivalent to moving $\omega$ in its Kähler class. Via pullback by a time-one Hamiltonian flow, we have
\begin{equation*}
  \phi_{f}^* \big(\omega + dJdf, J\big) = \big(\omega, \phi_f^*J\big).
\end{equation*}

It would seem natural to adopt the same construction here; that is detailed in Szekelyhidi's paper \cite{Sze1}. However, as $J$ is not integrable, we run into an obstacle: the obtained almost complex structure $\phi_f^*J$ is not compatible with $\omega$.

\begin{lemme}
  The symplectic form
  \begin{equation*}
    \omega_f := \omega -dJ df
  \end{equation*}
  is $J$- invariant if, and only if, $J$ is integrable.
\end{lemme}

\begin{proof}
  On the one hand we have, for any $X$, $Y$,
  \begin{align*}
    (dJ df)(X,Y) &= X\cdot(J df(Y)) - Y\cdot (J df(X)) - J df([X,Y])\\
    &= -X\cdot(J Y)\cdot f + Y\cdot(J X) \cdot f +J[X,Y]\cdot f.
  \end{align*}
  On the other hand,
   \begin{align*}
    (dJ df)(J X,J Y) &= J X\cdot df(Y) - J Y\cdot  df(X) - J df([J X,J Y])\\
    &= -J X\cdot Y\cdot f + J Y\cdot X \cdot f +J[J X,J Y]\cdot f.
   \end{align*}
  As a consequence, the $J$-anti-invariant parf of $dJ df$ is
  \begin{equation*}
    (dJ df)(X,Y) - (dJ df)(J X,J Y) = -4J N_{J}(X,Y),
  \end{equation*}
  where $N_J$ denotes the \emph{Nijenhuis tensor} of the almost-complex structure $J$:
  \begin{equation*}
    N_J(X,Y) = \frac14 \left([JX,JY] - J[JX, Y] - J[X,JY] -[X,Y] \right),
  \end{equation*}
  which, by the celebrated Newlander-Niremberg theorem, vanishes iff $J$ is integrable.
\end{proof}

\medskip

Thus, in the case where $J$ is not integrable, we rather use the exponential map construction, which does not move $J$ in the complexified orbits, but does retain the complexified action at the infinitesimal level:
\begin{equation*}
  \frac{d}{dt}\bigg|_{t=0} J_t  = \frac{d}{dt}\bigg|_{t=0} \exp(-t \mathcal{L}_{X_f}J)J = J\mathcal{L}_{X_f}J,
\end{equation*}
which coincide $JP(f)$ obtained earlier.

\subsection{The Hermitian scalar curvature}\label{sec:hermscalcurv}
There are several competing notions of curvature on the almost-Kähler manifold $(V, \omega, J)$. We now discuss them and pick the most natural choice; more details can be found in Apostolov and Draghici's survey \cite{ApoDra}.

\medskip

First, one can consider the different Riemannian curvature tensors derived from the metric $g_J$: the Riemannian curvature tensor $\text{Rm}_{g_J}$, the Ricci curvature $\Ric_{g_J}$ and the scalar curvature $s_{g_J}$. From these, one can define the Ricci form $\rho := \Ric_{g_J}(J\cdot, \cdot)$. In the Kähler case, the complex structure is parallel, which add symmetries to $\text{Rm}$, and one can show that the Ricci form is closed of type $(1,1)$, and that its cohomology class is exactly the first Chern class of $V$. However, since $D J$ is not assumed to vanish, where $D$ denotes the Levi-Civita connection of $g_J$, the Ricci form is not necessarily closed or $J$-invariant; in particular, it is not a representant of the cohomology class $2\pi c_1(V)$.

\medskip

On the other hand, the almost complex structure $J$ allows us to see each tangent space $T_pV$ as a complex vector space. We will denote the resulting complex bundle by $(TV, J)$. It is identified $T^{1,0}V$ via
\begin{align*}
  X \in (TV, J) &\mapsto X^{1,0}:=\frac12(X - iJX) \in T^{1,0}V \subset TV \otimes \C \\
  Z + \bar Z &\mapsfrom Z
\end{align*}

We endow $(TV, J)$ with a Cauchy-Riemann operator defined by

\begin{equation*}
\delbar^{(TV,J)}_X Y = 2 \mathfrak{Re}\big([X^{0,1}, Y^{1,0}]^{1,0} \big)
\end{equation*}
which, in terms of the Levi-Civita connection of $g_J$, rewrites
\begin{equation*}
\delbar^{(TV,J)}_X Y = \frac12(D_X Y + J D_{J X} Y) - \frac12 J(D_X J)Y.
\end{equation*}

Together with the Hermitian inner product $h_J = \frac12(g_J - i\omega)$, this operator determines a Chern connection $\nabla^J$ on $TV$ such that $\nabla^{J,(0,1)} = \delbar^{(TV,J)}$. 
Since the almost Kähler structure is not assumed to be integrable, the Chern connection does not necessarily coincide with the Levi-Civita connection. Instead, both are related by
\begin{equation*}
\nabla_X Y = D_X Y - \frac12 J(D_XJ)Y.
\end{equation*}
\textit{Remark: }The torsion of this Chern connection is given by the Nijenhuis tensor $N_{J}$.

\medskip

The top exterior power $K^*_J :=\Lambda^m (TV,J)$, called the \emph{anticanonical bundle}, inherits a Hermitian product and a Hermitian connection from this construction.
Then, the curvature of the Chern connection on $K^*_J$ is of the form $i\rho^\nabla$ where  $\rho^\nabla$ is a real, closed 2-form, and moreover, is a representant of $2\pi c_1(V)$. We call it the \textit{Hermitian Ricci form}. 
\smallskip

The \textit{Hermitian scalar curvature} $s^\nabla$ is then defined to be its trace with respect to $\omega$:
\begin{equation*}
s^\nabla = 2\Lambda \rho^\nabla.
\end{equation*}

On a Kähler manifold, i.e. when the almost complex structure is integrable, all those notions of Ricci and scalar curvature coincide. To express their relationship in the almost-Kähler setting, we need to introduce yet another notion of curvature. Observe that the (4,0)-Riemannian curvature tensor $\text{Rm}_{g_J}$ can be identified to a symmetric endomorphism $\Lambda^2V \rightarrow  \Lambda^2V$ via
\begin{equation*}
 \text{Rm}_{g_J}(\alpha \wedge \beta)(X,Y) := \text{Rm}_{g_J}(\alpha^\sharp, \beta^\sharp, X,Y).
 \end{equation*}
The \textit{twisted Ricci form}, or \textit{\textasteriskcentered-Ricci form}, is then defined as the image of the symplectic form by this endomorphism:
\begin{equation*}
\rho^* = R_{g_J}(\omega),
\end{equation*}
and its trace with respect to $\omega$ is the \textit{\textasteriskcentered-scalar curvature} :
\begin{equation*}
s^* = 2\Lambda\rho^* = 2(R_{g_J}(\omega), \omega).
\end{equation*}
Then we have the following identites, which are proven in \cite{ApoDra}.

\begin{prop}\label{curvCompar}
The Riemannian, Hermitian and twisted Ricci form are related as follows:
\begin{align*}
  \rho^\nabla(X,Y) &= \rho^*(X,Y) - \frac14 \text{tr}(J D_XJ \circ D_YJ), \\
  \rho^*(X,Y) &= \frac12(\Ric_{g_J}(J X,Y)-\Ric_{g_J}(X, J Y)) + \frac12 ((D D^*J)X, Y).
\end{align*}
As far as the scalar curvatures are concerned, we have
\begin{align*}
  s^\nabla = s_{g_J} + \frac12 |D J|^2 = s^* - \frac12 |D J|^2 = \frac12(s_{g_J}+s^*).
\end{align*}
In this last formula, the norm of $D J$ is given by $|D J|^2 = -\frac12 \sum_i\text{tr}\,(D_{e_i} J \circ D_{e_i}J)$, with $\{e_i\}_i$ a local orthonormal frame for $g_J$.
\end{prop}

In the almost Kähler context, the Hermitian Ricci form and the Hermitian scalar curvature are natural substitutes to their Riemannian counterparts.

We will thus use $s^\nabla$ as a generalization to our context of the Riemannian scalar curvature. Of course, the anticanonical bundle and Chern connection, hence the Hermitian scalar curvature depends on the almost complex structure we use on $V$. Hence, we will be interested in the operator
\begin{align*}
s^\nabla : \mathcal{AC}_{\omega} &\longrightarrow \mathcal{C}^\infty(V)\\
J &\longmapsto s^\nabla(J).
\end{align*}

\paragraph{First variation of \texorpdfstring{$s^\nabla.$}{the Hermitian scalar curvature.}}

The first variation of the Hermitian scalar curvature operator with respect to $J \in \mathcal{AC}_{\omega}$ is given by the following formula, proven by Mohsen in his Master thesis \cite{Moh}:

\begin{prop}\label{mohsenformula} Define a curve $J_t$ in $\AC_\omega$ by 
\begin{equation*}
J_t = \exp(-ta)J\exp(ta),
\end{equation*}
for $a \in \mathcal L _\omega$ anticommuting to $J$, and set
\begin{equation*}
  \dot J = \frac{d}{dt}\bigg|_{t=0} J_t
\end{equation*}
the tangent vector at $t=0$. 
Then the first variation of the Hermitian scalar curvature along the curve $J_t$ is given by:
  \begin{equation}\label{varscalcurv}
\frac{d}{dt}\bigg|_{t=0}s^\nabla(J_t) = \Lambda d(\delta\dot{J})^\flat = -\delta J (\delta\dot{J})^\flat,
\end{equation}
where the codifferential $\delta$ and the musical operator $\flat$ are taken with respect to the metric $g_J = \omega_\varepsilon(\cdot, J\cdot)$.
\end{prop}

\noindent\emph{Remark: }  Recall that the vector field $\delta \dot J$ is given in a local orthonormal frame $(e_i)_i$ for $g$ by
\begin{equation*}
\delta \dot J = - \sum (D^g_{e_i} \dot J)(e_i).
\end{equation*}

\begin{proof} We follow the proof given in Chapter 9 in \cite{Gau}. 

We denote by $g_t$, $h_t$ the Riemannian metric and Hermitian inner product on $(TV,J_t)$. Then the isomorphism
\begin{equation*}
  \exp(-ta) : (TV, J) \rightarrow (TV, J_t)
\end{equation*}
preserves $\omega$, hence induces an isomorphism of Hermitian line bundles between $(K^*_{J}, h)$ and $(K^*_{J_t}, h_t)$.

The strategy is to first compute the connection 1-form $\alpha_t$ of the Chern connection on $(K^*_{J_t}, h_t)$. Then the Hermitian Ricci curvature is given by $\rho^{\nabla^{J_t}} = - d\alpha_t$ , and taking the trace, we get the Hermitian scalar curvature $s^{\nabla^{J_t}} = 2\Lambda_t d\alpha_t$. Thus, we need only compute $\dot \alpha :=\dfrac{d}{dt}\bigg|_{t=0} \alpha_t$.
\ \\

\medskip

We wish to compute $\dot \alpha$ in terms of $\dot J$. Let $(Z_1, \dots, Z_m)$ be a local orthonormal frame for $(TV, J, h_J)$. That is,
\begin{equation*} 
h_J(Z_i, Z_j) = \delta_{ij} \Leftrightarrow
  \begin{dcases}
    g_J(Z_i, Z_j) = 2\delta_{ij},\\
    \omega(Z_i, Z_j) = 0.
  \end{dcases}
\end{equation*}

Then $\left\{Z_j^t : =exp(-ta)Z_j \right\}_{j=1\dots m}$ is an orthonormal frame for $(TV, J_t, h_{J_t})$. In this frame, the connection 1-form $\alpha_t$ is given by
\begin{equation*}
  \alpha_t(X) = -i\sum_j h_t(\,\nabla_X^{J_t}Z_j^t, Z_j^t ).
\end{equation*}
We split $\nabla^{J_t}$ into its (0,1) and (1,0) parts and observe that
\begin{equation*}
  h_t((\nabla^{J_t})^{(1,0)}X, Y) = -h_t(X, (\nabla^{J_t})^{(0,1)}Y)
\end{equation*}
thus
\begin{equation*}
  \alpha_t(X) = -i\sum_j h_t(\,(\nabla_X^{J_t})^{(0,1)}Z_j^t, Z_j^t) -h_t(Z_j^t,(\nabla_X^{J_t})^{(0,1)}Z_j^t)
\end{equation*}
Recall that the (0,1) part of $\nabla^{J_t}$ is $\delbar^{(TV, J_t)} $. Thus,
\begin{align*}
  \alpha_t(X) &= -i\sum_j h_t(\,\delbar_X^{(TV, J_t)}Z_j^t, Z_j^t) -h_t(Z_j^t,\delbar_X^{(TV, J_t)}Z_j^t)\\
  &=-\sum_j\omega(\delbar_X^{(TV, J_t)}Z_j^t,Z_j^t)\\
  &= \sum_j \omega(\exp(ta)\delbar_X^{(TV, J_t)} \exp(-ta)Z_j, \exp(ta)Z_j^t)) \\
   &= \sum_j \omega(\exp(ta)\delbar_X^{(TV, J_t)} \exp(-ta)Z_j, Z_j))
\end{align*}
Now, the Cauchy-Riemann operator $\delbar^{(TV, J_t)}$ is given by
\begin{align*}
  \delbar^{(TV, J_t)}_X Z:&=2\mathfrak{Re}([X^{0,1}, Z^{1,0}]^{1,0})\\
&=-\frac14(J_t\L_Z J_t + \L_{J_tZ}J_t)(X).
\end{align*}
As a consequence,
\begin{align*}
  \alpha_t(X) &= \frac14 \sum_j \omega(\exp(ta)J_t(\L_{\exp(-ta)Z_j} J_t)X, Z_j) +  \omega(\exp(ta)(\L_{J_t\exp(-ta)Z_j} J_t)X, Z_j)\\
  &= \frac14 \sum_j \omega(J\exp(ta)J(\L_{\exp(-ta)Z_j} J_t)X, Z_j) +  \omega(\exp(ta)(\L_{\exp(-ta)JZ_j} J_t)X, Z_j).
\end{align*}
We will now rewrite this in terms of the metric $g_J$ and its Levi-Civita connection $D$.
We will use the local frame $\{e_1,\dots,,e_{2m}\} := \dfrac1{\sqrt{2}} \{Z_1,\dots,Z_m, J Z_1,\dots,J Z_m\}$; in this frame, the previous expression rewrites
\begin{equation*}
  \alpha_t(X) = -\frac12 \sum_k g_J(\exp(ta)J(\L_{\exp(-ta)e_k} J_t)X, e_k).
\end{equation*}
We may express the Lie derivative of $J_t$ in terms of $D$:
\begin{align*}
  (\L_{\exp(-ta)e_k} J_t)X &= (D_{\exp(-ta)e_k}J_t)X + \big[D(\exp(-ta)e_k), J_t\big](X) \\
  &= (D_{\exp(-ta)e_k}J_t)X + D_{J_tX}(\exp(-ta)e_k) - J_tD_X(\exp(-ta)e_k).
\end{align*}
Hence, using $\exp(ta)J_t = J\exp(ta)$, we get
\begin{align*}
  \alpha_t(X) &= -\frac12\sum_k g_J(\exp(ta)(D_{\exp(-ta)e_k} J_t)X, e_k)\\
  &+\frac12\sum_k g_J(\exp(ta)D_{J_tX}(\exp(-ta)e_k), e_k)\\
  &-\frac12\sum_k g_J(J \exp(ta)D_{X}(\exp(-ta)e_k), e_k).
\end{align*}
Taking the derivative with respect to t yields
\begin{align*}
  \dot \alpha(X) = \frac12 \sum_k\ & g_J(a(D_{e_k}J )X, e_k) - g_J((D_{ae_k}J )X, e_k) +g_J((D_{e_k}\dot J )X, e_k)\\
 +& g_J(aD_{JX}e_k, e_k) + g_J(D_{\dot J X}e_k, e_k) - g_J(D_{JX}(ae_k), e_k) \\
 -&\ g_J(JaD_Xe_k, e_k) +g_J(JD_X(ae_k), e_k).
\end{align*}
which rewrites

\begin{align*}
  \dot \alpha(X) &= \frac12(\delta \dot J)^{\flat} (X)\\
  &-\frac12 \sum_k\  g_J(a(D_{e_k}J )X, e_k) - g_J((D_{ae_k}J )X, e_k)\\
 &+\frac12 \sum_k g_J((D_{JX}a)e_k, e_k) \\
 & -\frac12 \sum_kg_J(D_{\dot J X}e_k, e_k) \\
 &+\frac12 \sum_k\ g_J(J(D_Xa)e_k, e_k)
 \end{align*}
that is
\begin{align*}
  \dot \alpha(X) &= \frac12(\delta \dot J)^{\flat} (X)\\
  &-\frac12 \sum_k\  g_J(a(D_{e_k}J )X, e_k) - g_J((D_{ae_k}J )X, e_k)-g_J(D_X(ae_k, e_k)\\
 &+\frac12 \sum_k g_J((D_{JX}a)e_k, e_k) \\
 & -\frac12 \sum_kg_J(D_{\dot J X}e_k, e_k) \\
 &+\frac12 \sum_k g_J((D_XJa)e_k, e_k)
 \end{align*}

The first term $\frac12(\delta \dot J)^{\flat} (X)$ is what we expect. The other terms vanish, for the following reasons:
\begin{itemize}
\item Each $e_k$ has norme 1, thus $g_J(D_{\dot J X}e_k, e_k)= \frac12(\dot J X) (g(e_k, e_k) =0$.
\item Since $a$ and $J a$ anticommute to $J$, both these endormorphisms are trace-free, and so are $D_{J X}a$ and $D_X (Ja)$.
Thus, the terms $\sum_k g_J((D_{J X}a)e_k, e_k)$ and $\sum_k g_J((D_X Ja)e_k,e_k)$ vanish.
\item Finally, for any $k$, the sum
\begin{equation*}
 g_J((D_{e_k}J)(ae_k), X) + g_J((D_XJ)e_k, ae_k) + g_J((D_{ae_k}J)X, e_k)
\end{equation*}
vanishes, since for any $X,Y,Z$
\begin{equation*}
   g_J((D_{Y}J)(Z), X) + g_J((D_XJ)Y, Z) + g_J((D_{Z}J)X, Y) = d\omega(X,Y,Z) =0.
\end{equation*}
\end{itemize}

Thus we get
\begin{equation*}
  \frac{d}{dt}\bigg|_{t=0}\rho^\nabla(J_t) = d\dot \alpha = \frac12d(\delta \dot J)^{\flat} (X).
\end{equation*}

To get the variation, we need to take the trace. Howevern we must be careful: $\Lambda_t$ depends on $t$. 
However, we have, for any 1-form $\alpha$,
\begin{equation*}
  2\Lambda_td \alpha = -\delta_tJ_t \alpha,
\end{equation*}
and $\delta_t J_t$ actually does not depend on $t$. Indeed,by definition, we have for any smooth function $f$ and 1-form $\alpha$, 
\begin{equation*}
  \int_V(\delta_t \alpha)f \omega^m = \int_V \langle \alpha,df \rangle_t \omega^m = \int_V \alpha(\text{grad}_tf) \omega^m.
\end{equation*}
Thus,
\begin{equation*}
  \int_V(\delta_tJ_t \alpha)f \omega^m = \int_V \langle J_t\alpha,df \rangle_t \omega^m = -\int_V \alpha(J_t\text{grad}_tf) \omega^m = \int_V \alpha(X_f) \omega^m.
\end{equation*}

As a consequence, we have the announced result:
  \begin{equation*}
\frac{d}{dt}\bigg|_{t=0}s^\nabla(J_t) = \Lambda d(\delta\dot{J})^\flat = -\delta J (\delta\dot{J})^\flat.
\end{equation*}

\end{proof}

This results has other interesting consequences. For instance, if $J_1$ and $J_2$ are in $\AC_\omega$, then we get
\begin{align*}
  \rho^{\nabla_{J_1}} -  \rho^{\nabla_{J_2}} &= d \alpha_{J_1} - d \alpha_{J_2}\\
  &= -\frac12 d \left(\int_0^1 (\delta_t \dot J)^{\flat_t} dt \right)
\end{align*}
thus belong to the same de Rham class, the \emph{first Chern class} of the symplectic manifold $(V, \omega)$.

Moreover, if one defines the total Hermitian scalar curvature as
\begin{equation*}
  S^\nabla = \int_V s^\nabla \,vol_g,
\end{equation*}
then it is constant on $\AC_\omega$, as
\begin{align*}
  S^{\nabla_{J_1}} - S^{\nabla_{J_2}} = - \int_V \Lambda d \left(\int_0^1 (\delta_t \dot J)^{\flat_t} dt \right)\,vol_g = 0.
\end{align*}
This goes to say that the Hermitian scalar curvature on $\AC_\omega$ is the correct analogue in our context of the scalar curvature on a fixed Kähler class.
As an aside, note, we may push this analogy further and define a Hermitian Calabi functional by
\begin{align*}
  \mathcal{C} : \AC_\omega &\rightarrow \R \\
  J \mapsto \int_V s^\nabla(J)^2 vol_g,
\end{align*}
whose critical points are called extremal almost-Kähler metric and verify a similar condition as the extremal Kähler metrics. Such extremal almost Kähler metrics have been studied by Lejmi in \cite{Lej1}.

\paragraph{Relation to the Lichnerowicz operator.}
Using this formula, we can now compute the linearisation of the operator that will appear in the gluing construction, which is the composition of $s^\nabla$ with the map $f \mapsto J_f$ introduced in \eqref{eq:defJf}. In particular, we are interested with how it relates to the linearisation of the (riemannian) scalar curvature on a Kähler manifold.

\medskip

Recall that, on a Kähler manifold, the following formula holds:
\begin{equation*}
	\frac{d}{dt}\bigg|_{t=0}s(\omega + i \del\delbar f) = -2\delta\delta D^- df + (ds, df)= \frac12 \Delta^2 f + (2i\del\delbar f, \rho).
\end{equation*}
On a constant scalar curvature Kähler manifold, this reduces to the \emph{Lichnerowicz operator}
\begin{equation*}
\Lich f = (D^-d)^*D^-df = \delta\delta D^- df = \frac12 \Delta^2 f + \delta(\Ric(df)).
\end{equation*}

Choose $J \in \AC_\omega$ so that $(V,J,\omega)$ is almost-Kähler. We have
\begin{equation*}
	\frac{d}{dt}\bigg|_{t=0} J_{tf} = J\mathcal{L}_{X_f}J,
\end{equation*}
thus we want to compare
\begin{equation*}
	L:f \mapsto -\delta J (\delta( J\mathcal{L}_{X_f}J))^\flat
\end{equation*}
 to $\Lich$ in an attempt to translate its good regularity properties to our context.
\medskip

The main calculation is the following
\begin{prop}\label{prop:calcLinearisation}
Let $f \in C^{3,\alpha}(V)$. Then the following holds:
	\begin{equation*}
		J\delta( J\mathcal{L}_{X_f}J))^\flat = \Delta_{g} df - 2  \Ric(\gradV f, \cdot) + E f,
	\end{equation*}
	where the error term $E$ is given, in an orthornormal basis for $g$ of the form by
	\begin{equation}\label{errorterm}
	\begin{aligned}
		E f(Y) = &\sum_{i} df((D^2_{e_i, J Y}J)e_i) + 2Ddf(e_i, J (D_YJ)e_i)
	\end{aligned}
	\end{equation}
	in an orthonormal frame $\{e_1,\dots, e_{2m}\} =\dfrac1{\sqrt{2}}\{Z_1,\dots,Z_m,J Z_1,\dots,J Z_m\}$ on $(TV, g)$.
\end{prop}

\begin{proof}
	The first thing we use is the following rewriting of $\dot J$:
	\begin{equation}\label{Jdotv2}
	\begin{aligned}
		\dot J &= J \L_{X_f}J \\
		&= \L_{J X_f}J - 4 N_{J}(X_f, \cdot)\\
		&= \L_{\gradV f}J - 4 N_{J}(X_f, \cdot).
	\end{aligned}
	\end{equation}
	We will compute $\delta(\L_{\gradV f}J)$ and $\delta N_{J}(X_f, \cdot)$ separately.
	\medskip

	For the first, let $\psi_t$ be the flow of $\gradV f$. Then
	\begin{equation*}
		\L_{\gradV f} J = \frac{d}{dt}\bigg|_{t=0} \psi_t^* J.
	\end{equation*}
	Now, $(V, J, \omega)$ is an almost Kähler manifold, thus $\delta J =0$, which implies
	\begin{equation*}
			\psi_t^*(\delta J)=\delta^{\psi_t^*g	}\psi_t^*J=0.
		\end{equation*}
	Differentiating this equation at 0 with respect to $t$, we get
	\begin{equation*}
		\delta \L_{\gradV f} J = - \frac{d}{dt}\bigg|_{t=0} \big(\delta^{\psi_t^*g}\big)J.
	\end{equation*}
	To rewrite this expression, we use the following, proven by Minerbe in his thesis \cite{Min} (Lemma 3.19):
	\begin{equation}\label{minerbe}
		\frac d{dt}\bigg|_{t=0} D^{\psi_t^*g}_X Y = \text{Rm}^{g}(X, \gradV f)Y - D^2_{X,Y}\gradV f.
	\end{equation}
	We choose an orthonormal basis $\{e_i\}_{i=1\dots2m}$ of $(TV, g)$ of the form $\dfrac1{\sqrt2}\{Z_1,\dots, Z_m, J Z_1, \dots, J Z_m\}$, with $\{Z_i\}_i$ an orthonormal basis for the complex vector bundle $(TM, J)$ (as in the proof of Proposition \ref{mohsenformula}). In such a basis
	\begin{equation*}
		\delta^{\psi_t^*g}J = -\sum_{i,j}(\psi_t^*g)^{ij}D_{e_i}^{\psi^*_tg}J(e_j),
	\end{equation*}
  where $(\psi_t^*g)^{ij}$ denotes the $(i,j)$-coefficient of the inverse of the matrix $(\psi_t^*g(e_k,e_l))_{k,l}$.
	Using \eqref{minerbe}, we get
	\begin{align*}
	 	\frac{d}{dt}\bigg|_{t=0} D_{e_i}^{\psi^*_tg}J(e_j) &=  \frac{d}{dt}\bigg|_{t=0}\big(D_{e_i}^{\psi^*_tg}(J e_j) - J D_{e_i}^{\psi^*_tg}e_j\big)\\
	 	&= \text{Rm}(e_i, \gradV f)J e_j - D^2_{e_i, J e_j} \gradV f -J\text{Rm}(e_i, \gradV f)e_j + J D^2_{e_i, e_j}\gradV f.
	\end{align*}
	On the other hand, since we have chosen an orthonormal basis for $g$, $(\psi_t^*g)^{ij}_{|t=0} = \delta_{ij}$, thus
	\begin{equation*}
		\frac{d}{dt}\bigg|_{t=0} (\psi_t^*g)^{ij} =- \frac{d}{dt}\bigg|_{t=0} (\psi_t^*g)_{ij} = -\L_{\gradV f}g(e_i, e_j) = -2Ddf(e_i, e_j).
	\end{equation*}
	Thus,
	\begin{equation}\label{deltaLJ}
	\begin{aligned}
		-\delta \L_{\gradV f} J = &-\sum_i \text{Rm}(e_i, \gradV f)J e_i - D^2_{e_i, J e_i}  \gradV f -J\text{Rm}(e_i, \gradV f)e_i + J D^2_{e_i, e_i}\gradV f\\
		&+\sum_{i,j} 2Ddf(e_i, e_j)D_{e_i}J(e_j).
	\end{aligned}
	\end{equation}
	Now, using Bianchi's identity,
	\begin{align*}
		\text{Rm}(e_i, \gradV f)J e_i &= -\text{Rm}(\gradV f, J e_i)e_i - \text{Rm}(J e_i, e_i) \gradV f\\
		&= \text{Rm}(J e_i,\gradV f)e_i - \text{Rm}(J e_i, e_i) \gradV f
	\end{align*}
	Now, our choice of basis gives
	\begin{equation*}
		\sum_i \text{Rm}(e_i, \gradV f)J e_i = - \sum_i \text{Rm}(J e_i,\gradV f)e_i,
	\end{equation*}
	thus
	\begin{equation*}
		\sum_i \text{Rm}(e_i, \gradV f)J e_i = \frac12 \sum_i \text{Rm}(e_i, J e_i) \gradV f.
	\end{equation*}
	On the other hand, still thanks to the form of the local frame $\{e_i\}$,
	\begin{equation*}
		\sum_i D^2_{e_i, J e_i}\gradV f = \frac 12 \sum_i \left(D^2_{e_i, J e_i}\gradV f  - D^2_{J e_i, e_i}\gradV f \right) =\frac12 \sum_i \text{Rm}(e_i, J e_i) \gradV f.
	\end{equation*}
	As a consequence, the first two terms in \eqref{deltaLJ} compensate one another. As for the remaining terms, we~use
	\begin{equation*}
		\sum_i \text{Rm}(e_i, \gradV f)e_i = -\Ric(\gradV f),
	\end{equation*}
	 thus \eqref{deltaLJ} rewrites
	\begin{equation*}
		\delta \L_{\gradV f} J = -J D^*D \gradV f -J\Ric(\gradV f) -\sum_{i,j} 2Ddf(e_i, e_j)D_{e_i}J(e_j).
	\end{equation*}
Using Bochner's formula on 1-foms, this rewrites
\begin{equation*}
	(\delta \L_{\gradV f} J)^\flat = J\Delta df - 2 \Ric(\gradV f, J \cdot) -\sum_{i} 2D_{e_i}df\circ D_{e_i}J.
\end{equation*}
\medskip

We still have the second term of \eqref{Jdotv2} to deal with. We need to compute
\begin{equation*}
	(\delta N_{J}(X_f, \cdot) )^\flat.
\end{equation*}
However, the Nijenhuis tensor rewrites as follows in terms of the Levi-Civita connection
\begin{equation*}
	g(N_{J}(X_f, X), Y) = \frac12 g(X_f, J (D_Y J)X).
\end{equation*}
Thus,
\begin{equation*}
	(\delta N_{J}(X_f, \cdot) )^\flat(Y) = \delta\alpha(Y)
\end{equation*}
where $\alpha(X,Y) :=  -\dfrac12 g(\gradV f, (D_Y J)X)$.
Hence
\begin{align*}
	(\delta N_{J}(X_f, \cdot) )^\flat(Y) &= -\sum_i D_{e_i}\alpha(e_i, Y) \\
	&= -\sum_i e_i\cdot(\alpha(e_i, Y)) - \alpha(D_{e_i}e_i, Y) - \alpha(e_i, D_{e_i}Y) \\
	&= \frac12\sum_i g(D_{e_i}\gradV f, (D_Y J)e_i) + g(\gradV f, (D^2_{e_i, Y}J)e_i).
\end{align*}

Moreover, observe that since $D_Y J$ is antisymmetric with respect to the metric $g$, while the Hessian $Ddf$ is symmetric, the first term must vanish. 
Indeed, in a basis that simultaneously diagonalises $Ddf$ and $g$, we see that
\begin{align*}
  \sum_i g(D_{e_i}\gradV f, (D_Y J)e_i) &=\sum_i Ddf(e_i,(D_Y J)e_i )\\
  & = \sum_i \lambda_i g(e_i,(D_Y J)e_i  ) \\
  & = - \sum_i \lambda_i g((D_Y J)e_i,e_i) \\
  &= - \sum_i g(D_{e_i}\gradV f, (D_Y J)e_i).
\end{align*}
As a consequence, we are left with
\begin{align*}
  J\delta( J\mathcal{L}_{X_f}J))^\flat(Y) &=J \delta(\L_{\gradV f}J)^\flat(Y) -4 J (\delta N_{J}(X_f, \cdot) )^\flat(Y) \\
  &= \Delta_{g} df(Y) - 2  \Ric(\gradV f, Y)-2\sum_{i} D_{e_i}df( (D_{e_i}J) Y) -2\sum_i df(D^2_{e_i, J Y}e_i),
\end{align*}

which is what we set out to demonstrate, provided $J$ act on 1-forms the usual way:
\begin{equation*}
	(J \alpha)(Y) = -\alpha(J Y).	
\end{equation*}
\end{proof}

The error term gives the quantity we will need to estimate when comparing the linearisation of our equation to model operators on $M$ and $X$. We can see it is directly related to the lack on integrability of $J$.

Applying the codifferential $\delta$ again, we see that
\begin{equation}
	L f = - \Delta_{g}^2f + 2\delta(\Ric(df)) + \delta E f,
\end{equation}
that is, the linearised operator is equal to the Lichnerowicz operator, plus an error term of order at most 3 in $f$. The coefficients of this error term depends on (derivatives of) $DJ$, which is comparable to the Nijenhuis tensor. 
As a consequence, $\L$ is an elliptic, 4th-order operator on the potential function~$f$.

\section{Darboux charts in the orbifold and the ALE space}\label{sec:Darboux}

When gluing together an orbifold with the resolutions of its singularities, holomorphic charts are usually used, to obtain a ``connected sum'' that is naturally a complex manifold. However, here the construction will not work in holomorphic charts, as the complex structures do not match on the ALE space $X$ and the Kähler orbifold $M$; the connected sum we will obtain will have no natural complex structure inherited from that of the orbifold.

To address that issue, we will work in Darboux charts instead, and endow the connected sum with a symplectic structure.

\subsection{On the orbifold.} 
Let $(M, J_M, \omega_M)$ be a Kähler orbifold of complex dimension 2, with singularities $p_1,\dots, p_\ell$ of type $\CZ$.
Let $p_i$ be a singular point of $M$. Then, there is a neighborhood $U_i$ of 0 in $\C^2$ and a map 
\begin{equation*}
	\phi_i : U_i \rightarrow M,
\end{equation*} 
such that $\phi_i(0) =p_i$ and $\phi_i$ induces an homeomorphism
\begin{equation*}
	\tilde \phi_i : \faktor{U_i}{\Z_2} \rightarrow \tilde{U_i}\subset M.
\end{equation*}

In such a chart, the Kähler form $\omega_M$ pulls back to a $\Z_2$ invariant, closed, nondegenerate 2-form $\omega_i$ on $U_i$.

Up to a linear transformation of the coordinates, we may assume that in this chart, at the point 0 we have
\begin{equation*}
	\omega_i(0) = \omega_0 :=\frac{\sqrt{-1}}2 \sum dz_k\wedge d\bar{z_k}.
\end{equation*}
Moreover we may arrange that the complex structure $J_M$ is also equal to the standard one $J_0$ at 0.

Now, since $\Z_2 \subset U(2)$, the standard symplectic structure $\omega_0$ on $U_i$ is also $\Z_2$ invariant. Thus we can use the equivariant version of the relative Darboux theorem, relatively to the point 0 where both 2-forms agree, to find an equivariant symplectomorphism 
\begin{align*}
	\psi : V_i \subset U_i &\rightarrow V_i \subset U_i,\\
	\psi^*\omega_i &= \omega_0.
\end{align*}
This is proven the usual way, by working $\Z_2$-equivariantly; the interested reader may consult \cite{DelMel}.

This symplectomorphism passes to the quotient modulo $\Z_2$ and, composed with $\phi_i$, provides an orbifold Darboux chart around $p_i \in M$.

Moreover, since $\omega_0(0)=\omega_i(0)$, working relatively to 0 we may assume that $d\psi(0) = I$, thus in this Darboux chart, the complex structure $J_M$ is equal to $J_0$ at $p$.

\medskip

\subsection{On the ALE manifold.} 
The second ingredient of the gluing construction is an ALE Kähler manifold $X$, with group at infinity $\Z_2$. We consider $X=T^*S^2$ endowed with the family of Ricci-flat Kähler metrics $(J_{X,\varepsilon}, g_{X,\varepsilon})$ that are described in the Annex. They are obtained when considering smoothings instead of the minimal resolution of the quotient singularity.  In spherical coordinates in $\R^4$, we have the following expression:
\begin{equation}\label{expressionStenzel}
\begin{aligned}
J_{X,\varepsilon}\dfrac{\del}{\del s} &= -\dfrac{2s}{\sqrt{s^4 -4}}X_3,\ J_S X_1 = -\sqrt{1-\frac4{s^4}}X_2\\
\frac1{\sqrt{2}\varepsilon}g_{X,\varepsilon}&= \left(1-\dfrac4{s^4}\right)^{-1}ds^2 + \dfrac{s^2}4 \left(1-\dfrac4{s^4}\right) \alpha_1^2 + \dfrac{s^2}4(\alpha_2^2 + \alpha_3^2),\\
\omega_{X,\varepsilon} &=  \sqrt{2}\varepsilon \left(\dfrac{s}{2\sqrt{1 - \frac{4}{s^4}}} \alpha_3 \wedge ds + \frac{s^2}4 \sqrt{1- \frac{4}{s^4}}\  \alpha_2\wedge \alpha_1 \right)
\end{aligned}
\end{equation}
where $s$ is the radius function of $\R^4$, and the $\alpha_i$'s are a basis of invariant 1-forms on $S^3$, verifying $d\alpha_i = 2\alpha_j \wedge \alpha_k$ for any circular permutation (i,j,k) of (1,2,3), and the $X_i$'s are the associated dual basis.
Thus, $(J_{X,\varepsilon}, g_{X,\varepsilon})$ gives a Kähler structure on $T^*S^2$ that is ALE of order 4:

\medskip

\medskip

To endow $(X= T^*S^2, \omega_X)$ with a Darboux chart outside a compact, notice that
\begin{align*}
\omega_{X,\varepsilon} &=\sqrt{2}\varepsilon dd^c_{J_{X, \varepsilon}}\left(\frac{s^2}2 \right) =  \sqrt{2}\varepsilon \left(\dfrac{s}{2\sqrt{1 - \frac{4}{s^4}}} \alpha_3 \wedge ds + \frac{s^2}4 \sqrt{1- \frac{4}{s^4}}\  \alpha_2\wedge \alpha_1 \right) \\
 &= f_\varepsilon'(s)\, \alpha_3 \wedge ds + f_\varepsilon(s)\, \alpha_2\wedge \alpha_1.
 \end{align*}
 où
 \begin{equation*}
 	f_\varepsilon(s) = \sqrt{2}\varepsilon  \frac{s^2}4 \sqrt{1- \frac{4}{s^4}}
 \end{equation*}

Thus, setting $\dfrac{r^2}{2} = f_\varepsilon(s)$ gives a radial change of coordinate that provides a Darboux chart outside a compact set in $X$. Moreover this change of variable gives us the same ALE fall-off rate. 
Indeed, straightforward computation gives, in these new coordinates:
\begin{align*}
   \omega_{X,\varepsilon} &= r\, \alpha_3 \wedge dr + \frac{r^2}2\, \alpha_2 \wedge \alpha_1 = \omega_0;\\
   g_{X,\varepsilon} &= \left(1 + \frac{\varepsilon^2}{r^4}\right)^{-\frac12}\, dr^2 + \frac{r^2}4 \left( 1 + \frac {\varepsilon^2}{r^4}\right)^{-\frac12} \, \alpha_1^2 +  \frac{r^2}4 \left( 1 + \frac {\varepsilon^2}{r^4}\right)^{\frac12}(\alpha_1^2 + \alpha_3^2) \\
   J_{X,\varepsilon} \frac{\del}{\del r} &= - \frac{2r}{\sqrt{r^4 + \varepsilon^2)}} X_3 \\
   J_{X,\varepsilon}X_1 &= -  \left( 1 + \frac {\varepsilon^2}{r^4}\right)^{-\frac12}X_2.
 \end{align*}

 From these expressions, we see that the decay rate in this ALE Darboux chart is still 4:

\begin{equation}\label{eq:estimALE}
\begin{aligned}
    \del^k(J_0-J_{X,\varepsilon}) &= O(r^{-4-k})\\
    \del^k(g_0-g_{X,\varepsilon}) &= O(r^{-4-k}).
\end{aligned}
\end{equation}

\begin{rem}
	Moreover, in this chart, we observe that as $\varepsilon$ goes to 0, the Kähler structure on $T^*S^2 \setminus S^2$ outside the zero section converges to the orbifold Euclidean structure in $\CZ$, in any $\mathcal C^k$ norm.
\end{rem}

\subsection{Symplectic connected sum.}

Using these charts on $M$ and $X$, we obtain a new manifold by a generalized connected sum construction, and that manifold will naturally be a symplectic one.
Since $M$ has isolated singularities, we can assume that the Darboux charts around each of them are disjoint. 

Define a function $\rho$ on $M$ that, in each such chart, is equal to the distance to the singularity $p_i$  and extend it smoothly to 1 on $M$. 

On $X$, we use the radius function $r$ in our ALE Darboux chart away from the zero section of $T^*S^2$. We extend it smoothly to 1 on a compact neighborhood of the zero section.

\medskip

Let $\varepsilon \in (0,\varepsilon_0)$ be a small gluing parameter, and let $r_\varepsilon :=\varepsilon^\beta$ for a $0<\beta <1$, $R_\varepsilon = \dfrac{r_\varepsilon}\varepsilon$. We identify the regions $\{\rho = 2r_\varepsilon\} \subset M$ and $\{r = 2R_\varepsilon\} \subset X$ via the homothety
\begin{align*}
  h_{\varepsilon^{-1}} : \{ \varepsilon \leq \rho \leq 1\} \subset M &\rightarrow  \{ 1 \leq r \leq \varepsilon^{-1}\} \subset X \\
  z &\mapsto w = \frac z \varepsilon.
\end{align*}

We perform this connected sum construction at each singularity $p_i$ to get a smooth compact manifold $M_\varepsilon$, which is naturally endowed with the symplectic form
\begin{equation*}
  \omega_ \varepsilon =
  \begin{dcases}
    \varepsilon^2h_{\varepsilon^{-1}}^*\omega_{X,\varepsilon} &\text{ on } \{\rho\leq 2r_\varepsilon\}, \\
    \omega_M &\text{ on } \{\rho \geq 2r_\varepsilon\}.
  \end{dcases}
\end{equation*}
The use of Darboux charts ensure that this 2-form is smooth, nondegenerate and closed.

\begin{rem}
  There is actually another degree of freedom that we do not use here. Indeed, we could make sense of the construction with a \emph{complex} nonzero parameter $\varepsilon$, which would be tantamount to introduce an action of $S^1$. 
\end{rem}

All the manifolds $M_\varepsilon$ are diffeomorphic to the minimal resolution $\hat M$ of the singularities $p_i$. 
Moreover, as advertised in the introduction, the region
\begin{equation*}
	M \setminus \cup_i B(p_i, 4r_\varepsilon)
\end{equation*}
is naturally included in each $M_\varepsilon$, allowing us to define: 

\begin{definition}\label{def:convCompM}
	Suppose that we have, for each $\varepsilon \in (0,\varepsilon_0)$, a (smooth) function $f_\varepsilon: M_\varepsilon \rightarrow \R$. Let $f_0: M \rightarrow \R$ be a function defined on the orbifold $M$. Let $K$ be a compact subset of $M^*$. There is $\varepsilon_1 >0$ such that for all $\varepsilon < \varepsilon_1$, $K \subset M \setminus \cup_i B(p_i, 4r_\varepsilon)$. Then, for all  $\varepsilon < \varepsilon_1$, $f_{|K}$ is defined on $M_\varepsilon$.
	We say that the sequence $(f_\varepsilon)_\varepsilon$ converges towards $f$ in $\mathcal C^k$ norm on the compact $K$ if
	\begin{equation*}
		\|f_{\varepsilon |K} - f_{|K} \|_{\mathcal C^k(K)}  \xrightarrow{\varepsilon \rightarrow 0} 0.
	\end{equation*}
\end{definition}

This definition extends to tensors on $M_\varepsilon$. Then, we see that the sequence of symplectic forms $(\omega_\varepsilon)_\varepsilon$ converges to the orbifold symplectic form $\omega_M$, in any $\mathcal C^k$ norm, on every compact set of $M^*$.

\medskip

Conversely, the compact set $\{r \leq R_\varepsilon \} \subset X$, after rescaling, is naturally included in a small region of $M_\varepsilon$. Thus we may define:
\begin{definition}\label{def:convCompX}
	Suppose that we have, for each $\varepsilon \in (0,\varepsilon_0)$, a (smooth) function $f_\varepsilon: M_\varepsilon \rightarrow \R$. Let $f_0: X \rightarrow \R$ be a function defined on the ALE manifold $X$. Let $K$ be a compact subset of $X$, then there is $\varepsilon_1 >0$ such that for all $\varepsilon < \varepsilon_1$, $K \subset \{r \leq R_\varepsilon \} \hookrightarrow M_\varepsilon$. Then, for all  $\varepsilon < \varepsilon_1$, $h_\varepsilon^* f_{\varepsilon|K}$ is defined on $X$.
	We say that the sequence $(f_\varepsilon)_\varepsilon$ converges towards $f$ in $\mathcal C^k$ norm on the compact set $K$ if
	\begin{equation*}
		\|h_\varepsilon^*f_{\varepsilon |K} - f_{|K} \|_{\mathcal C^k(K)}  \xrightarrow{\varepsilon \rightarrow 0} 0.
	\end{equation*}
\end{definition}

Moreover,
\begin{lemme}\label{lemme:cohomIndEps}
	The cohomology class $[\omega_\varepsilon]$ does not depend on $\varepsilon$.
\end{lemme}

\begin{proof}
	Notice that, on the orbifold $M$, in a contractile neighborhood of each $p_i$, the orbifold version of the local $\del\delbar$-lemma tells us that $\omega_M$ is exact. Thus, there is a 2-form $\bar \omega \in H^2(M^*, \R)$, where $M^* := M \setminus \{p_1, \dots, p_\ell\}$, and functions $\varphi_i$ supported in a neighborhood of each $p_i$, such that
	\begin{equation*}
		\omega_M = \bar \omega + i\sum_j \del\delbar \varphi_j. 
	\end{equation*}
	On the other hand, since $\omega_X = i\del\delbar u$ is exact (see Annex), from the definition of $\omega_\varepsilon$ we see that we may write
	\begin{equation*}
		\omega_\varepsilon = \bar \omega + \varepsilon^2 \sum_j \del\delbar (\gamma_j u)
	\end{equation*}
	for suitable cut-off functions $\gamma_j$.
\end{proof}

\begin{rem} A more general, Mayer-Vietoris-type argument, actually allows to identify $H^2(M, \R)$ to $\{ \alpha\in H^2(\hat M_, \R), \alpha\cdot S =0 \}$ via $H^2_c(M^*, \R)$, where $S$ corresponds to the zero section in $T^*S^2$.
\end{rem}

From here, using Moser's stability theorem (see for instance \cite{McDSal}, Theorem 3.17), we get
\begin{cor}\label{cor:symplectEquiv}
	The symplectic manifolds $(M_\varepsilon, \omega_\varepsilon)_{\varepsilon \in (0, \varepsilon_0)}$ are all symplectically equivalent.
\end{cor}

\begin{rem}
  As a consequence, we could actually work on a fixed symplectic manifold $(\hat M, \hat \omega)$. As a matter of fact, this is what we will do in Section \ref{sec: HamStat}. However, during the gluing construction, it is more practical for the analysis to keep track of the parameter $\varepsilon$ (for instance to use Definitions \ref{def:convCompM} and \ref{def:convCompX}).
\end{rem}

\section{Almost complex structures on \texorpdfstring{$M_\varepsilon$}{the gluing}}\label{sec:cpxmodif}

The next step is to endow $M_\varepsilon$ with an almost complex structure that is compatible with $\omega_\varepsilon$. We achieve this by gluing together the complex structures $J_M$ on $M$ and $J_X$ on $X$. As these manifolds have differing complex structures, making them compatible will come at the cost of integrability, thus we will only get an almost-complex structure on $M_\varepsilon$.

\subsection{On the orbifold $M$.}
 Recall that we are working in orbifold Darboux charts $(U_i, \phi_i)$ centered at each singularity $p_i$. In such a chart, $J_M$ is, of course, compatible with $\omega_M$, but so is $J_0$, the standard complex structure in $\C^2$.

Thus, according to the proposition \ref{propJcompatible}, there is a unique section $A$ of $\text{End}(TU_i)$, anticommuting with both $J_M$ and $J_0$, such that
\begin{equation*}
  J_M = \exp(A)J_0 \exp(-A).
\end{equation*}

Now, multiplying $A$ by a cut-off function on $M$, we will be able to transition smoothly from $J_M$ to $J_0$ in a neighborhood of the singularities.
We will lose integrability of the resulting almost complex structure in the process. On the other hand, if we can show that $J_M$ approaches $J_0$ close to each $p_i$, we may hope that the operation is not  too drastic.
\newline

Thus, we first need an estimate of $A$:

\begin{lemme}\label{lemme:diffJ0JM}
	In the orbifold Darboux coordinates $x=(x_k)_{k=1,\dots,4}$ described in paragraph \ref{sec:Darboux}, $J_M$ and $J_0$ coincide to first order:
	\begin{equation}\label{diffJMJ0}
		J_M(x) = J_0 + O(|x|^2).
	\end{equation}
	As a consequence, the endomorphism $A$ satisfies the following estimates:
	\begin{equation}\label{estimA}
	\begin{aligned}
		A &= O(|z|^2), \\
  		\del A&= O(|z|),\text{ and }\\
  		\del^k\! A &= O(1)\text{ for all } k \geq 2.
\end{aligned}
 \end{equation}
\end{lemme}

\begin{proof}
Recall that in the orbifold charts that we are using, we have arranged that $J_M(0) = J_0$. Thus, in these coordinates, a Taylor development of $J_M$ around 0 can be written
\begin{equation*}
	J_M(x)_i^j = (J_0)_i^j + (J_{(1)})^j_{ik}x_k + O(|x|^2).
\end{equation*}
The tensor $J_{(1)}$, whose coefficients are the first order coefficients in the development of $J_M$, is a local section of $\Lambda^1 U_i \otimes \text{End}(TU_i)$. However, as $\Z_2$ acts as a multiplication by -1 on $\Lambda^1 U_i \otimes \text{End}(TU_i)$ can only be $\Z_2$-invariant if it is zero. As both $J_M$ and $J_0$ are $\Z_2$-invariant, we obtain the estimate \eqref{diffJMJ0}.
Observing that
\begin{equation*}
	\begin{aligned}
		J_M - J_0 &= \exp(A)J_0 \exp(-A) - J_0 \\
		&= (\exp(2A) - I)J_0\\
  		&= O(|x|^2),
\end{aligned}
 \end{equation*}
 we get the desired estimate on $A$.
 Writing a Taylor development of $A$ and using again that $J_M(x) - J_0 = O(|x|^2)$ allows to get the estimate on the first derivative of $A$ near 0.

 Since $A$ is defined and smooth on $M$, we see that higher order derivatives are at worst bounded.
\end{proof}

\noindent\textit{Remark: } When performing gluing on a Kähler manifold, it is usual to work in holomorphic coordinates in which $\omega$ approaches the standard Kähler form $\omega_0$ on $\C^m$ to order 2. The existence of such a charts is actually a characterization of Kähler metrics. Here, we work in a Darboux chart instead, but we do retrieve an order two approximation, on the complex structure instead of the symplectic form.
\newline

Now recall that $r_\varepsilon = \varepsilon^\beta$ is our chosen gluing radius; for $\varepsilon$ small enough, $\{ \rho \leq 4r_\varepsilon \}$ is contained in the Darboux chart around each $p_i$.

Let $\chi_1 : \R \rightarrow \R$ be a smooth cutoff function, such that
\begin{equation*}
	\chi_1(x) =
	\begin{dcases}
		0 \text{ if } x \leq 2 + \eta,\\
		1 \text{ if } x \geq 4
	\end{dcases}
\end{equation*}
where $\eta$ is very small; its only purpose is to provide some leeway and ensure that all derivatives will match when performing the gluing. Set
\begin{equation*}
	\chi_{r_\varepsilon} := \chi_1\!\left(\dfrac{\rho}{r_\varepsilon}\right).
\end{equation*}

We define an almost complex structure $J_{r_\varepsilon}$ on $M$ by
\begin{equation*}
  J_{r_\varepsilon} = \exp(\chi_{r_\varepsilon}A) J_0 \exp(-\chi_{r_\varepsilon}A).
\end{equation*}

In particular,
\begin{equation*}
   J_{r_\varepsilon}=
   \begin{dcases}
     J_0 \text{ if } \rho \leq 2r_\varepsilon, \\
     J_M \text{ if } \rho \geq 4r_\varepsilon.
   \end{dcases}
\end{equation*}

Moreover, using Lemma \ref{lemme:diffJ0JM} in the ``annulus'' $\{2r_\varepsilon \leq \rho \leq 4r_\varepsilon\}$, we see that
\begin{equation}\label{JMJ0}
\begin{aligned}
  J_{r_\varepsilon} - J_0 &= O(r_\varepsilon^2), \\
  \del(J_{r_\varepsilon} - J_0 ) &=O(r_\varepsilon).
\end{aligned}
\end{equation}

The first estimate results directly from the lemma. For the second, observe that
\begin{equation*}
  J_{r_\varepsilon} - J_0  = \big(\exp \big(2\chi_{r_\varepsilon}A\big) - I\big) J_0
\end{equation*}
thus first derivatives are of the form
\begin{equation*}
	\del(J_{r_\varepsilon} - J_0 ) = 2 (d\exp) (2\chi_{r_\varepsilon}A)( \del\chi_{r_\varepsilon}A +  \chi_{r_\varepsilon} \del A ) J_0.
\end{equation*}
To conclude, we use that in $\{2r_\varepsilon \leq r \leq 4r_\varepsilon\}$,
\begin{equation*}
\del \chi_{r_\varepsilon}=O(r_\varepsilon^{-1}).
\end{equation*}

\medskip 

The endomorphism $J_{r_\varepsilon}$ on $M$ is an almost complex structure, compatible with $\omega_M$ by construction. It is not an integrable complex structure; however, its Nijenhuis tensor is supported in the cutoff region $\{2r_\varepsilon \leq r \leq 4r_\varepsilon\}$.
We give an estimate of the Nijenhuis tensor $N_{J_{r_\varepsilon}}$, as it will appear in error terms down the road.

\begin{lemme}\label{lemme:estimNijM}
 The Nijenhuis tensor $N_{J_{r_\varepsilon}}$ of $J_{r_\varepsilon}$ verifies
 \begin{equation} \label{estimNijenhuisM}
   N_{J_{r_\varepsilon}} =
   \begin{dcases}
      O(r_\varepsilon) \text{ in } \{2r_\varepsilon \leq r \leq 4r_\varepsilon\}, \\
      0 \text{ elsewhere.}
   \end{dcases}
\end{equation}
Moreover, its derivatives are bounded on $M$.
\end{lemme}

\begin{proof} Recall that we have the following expression for the Nijenhuis tensor:
  \begin{equation}\label{NijLC}
    N_{J_{r_\varepsilon}}(X,Y)= \frac12 J_{r_\varepsilon}( (D_{r_\varepsilon,Y}J_{r_\varepsilon})X -(D_{r_\varepsilon,X}J_{r_\varepsilon}) Y) \\,
  \end{equation}
  where $D_{r_\varepsilon}$ is the Levi-Civita connection associated to the Riemannian metric $g_{r_\varepsilon}:=\omega(\cdot, J_{r_\varepsilon}\cdot).$
Using this, we compute:
 \begin{align*}
   N_{J_{r_\varepsilon}}(X,Y) &= \frac12 (J_{r_\varepsilon}-J_0)( (D_{r_\varepsilon,Y}(J_{r_\varepsilon} - J_0)X -(D_{r_\varepsilon,X}(J_{r_\varepsilon}-J_0) Y) \\
   &+ \frac12  J_0( D_{r_\varepsilon,Y}(J_{r_\varepsilon} - J_0)X -D_{r_\varepsilon,X}(J_{r_\varepsilon}-J_0) Y) \\
   &+ \frac12 J_{r_\varepsilon}( (D_{r_\varepsilon,Y} J_0)X -(D_{r_\varepsilon,X}J_0) Y),
 \end{align*}
where $D_{r_\varepsilon}$ is the Levi-Civita connection associated with the metric $g_{r_\varepsilon} = \omega_M(\cdot, J_{r_\varepsilon}\cdot)$.
Using the estimate \eqref{JMJ0}, we see that the first term of this sum is an $O(r_\varepsilon^3)$ and the second one is an $O(r_\varepsilon)$. We need estimate the third term by comparing it with the Nijenhuis tensor of $J_0$, which vanishes. To do this, notice that
\begin{equation*}
  D_ {r_\varepsilon}J_0 = (D_0 + \Gamma_{r_\varepsilon}) J_0=\Gamma_{r_\varepsilon} J_0,
\end{equation*}
where $\Gamma_{r_\varepsilon}$ is expressed with the Christoffel coefficients of the metric $g_{r_\varepsilon}$, thus the first derivatives of the coefficients of $g_{r_\varepsilon}$. As a consequence, $\Gamma_{r_\varepsilon} J_0~=~O(r_\varepsilon)$.
\end{proof}

\subsection{On the ALE space $X$.}
 We proceed similarly on $X$. We work in the (family of) Darboux charts at infinity described in paragraph \ref{sec:Darboux}. In this chart, both $J_{X,\varepsilon}$ and $J_0$ are compatible with $\omega_{X,\varepsilon}$, thus there is a unique section $B_\varepsilon$ in $\L_{\omega_{X,\varepsilon}}$, anticommunting with both $J_0$ and $J_{X,\varepsilon}$, and such that
\begin{equation*}
  J_{X,\varepsilon} = \exp(B_\varepsilon)J_0 \exp(-B_\varepsilon).
\end{equation*}
Using our estimate \eqref{eq:estimALE}, the same calculations that we already performed on $M$ show that
\begin{align*}
	B_\varepsilon &= O(r^{-4}),\\
	\del^kB_\varepsilon &= O(r^{-4-k}).
\end{align*}

We perform the same kind of cutoff as we did on the orbifold. Let $\chi_2: \R \rightarrow \R$ be a smooth cutoff function, such that
 \begin{equation*}
 \chi_2(x) =
 \begin{dcases} 1 \text{ if } x \leq 1\\
  0 \text{ if } x\geq2-\eta.
 \end{dcases}
 \end{equation*}

 Recall that $R_\varepsilon = \faktor{r_\varepsilon}{\varepsilon} = \varepsilon^{\beta-1}$ is our gluing radius on the ALE space.
We define a cutoff function on $X$ by
\begin{equation*}
	\chi_{R_\varepsilon} := \chi_2\!\left(\dfrac{r}{R_\varepsilon}\right).
\end{equation*}

 If $\varepsilon$ is small enough, the region $\{r \geq R_\varepsilon\}$ is contained in the Darboux chart. We define an almost-complex structure on $T^*S^2$ by
\begin{equation*}
  J_{R_\varepsilon} = \exp(\chi_{R_\varepsilon}B_\varepsilon) J_0\  \exp(-\chi_{R_\varepsilon}B_\varepsilon).
\end{equation*}

By definition,
\begin{equation*}
  J_{R_\varepsilon}=
  \begin{dcases}
    J_{X,\varepsilon} \text{ on } \{r \leq R_\varepsilon\} \\
    J_0 \text{ on } \{r \geq 2R_\varepsilon\}.
  \end{dcases}
\end{equation*}

As before, our estimate on $B_\varepsilon$ and choice of cutoff ensures that the difference between $J_{R_\varepsilon}$ and $J_0$ becomes small when $\varepsilon$ goes to zero. More precisely:
\begin{equation}\label{comReps}
\begin{aligned}
  J_{R_\varepsilon} - J_0 &= O(R_\varepsilon^{-4}), \\
  \del^k(J_{R_\varepsilon} - J_0 ) &=O(R_\varepsilon^{-4-k}).
\end{aligned}
\end{equation}

As before, $J_{R_\varepsilon}$ is a compatible almost complex structure on $X$, compatible with $\omega_X$ in the Darboux chart. However, once again, it is not integrable. Its Nijenhuis tensor is supported in $\{R_\varepsilon \leq \rho_x \leq 2R_\varepsilon\}$. The computation done on the orbifold translates directly to this case and we see that $N_{J_{R_\varepsilon}}$ verifies:

\begin{lemme}\label{lemme:estimNijX}
The Nijenhuis tensor of $J_{R_\varepsilon}$ verifies, for any $k\geq 0$,
 \begin{equation} \label{estimNijenhuisX}
   \del^k N_{J_{R_\varepsilon}} =
   \begin{dcases}
   O(R^{-5-k}_\varepsilon) \text{ on } \{R_\varepsilon \leq r\leq 2R_\varepsilon\} \\
   0 \text{ elsewhere.}
  \end{dcases}
 \end{equation}
\end{lemme}

\begin{proof} The proof is the same as Lemma \ref{lemme:estimNijM}, and relies on the expression \eqref{NijLC} for the Nijenhuis tensor. We apply it this time to the Levi-Civita connection associated with the metric $g_{R_\varepsilon}= \omega_X(\cdot,J_{R_\varepsilon}\cdot)$. The computation then translates directly to this case, using \eqref{estimALE} for the estimation of the Christoffel symbols.
\end{proof}

\subsection{The approximate solution}\label{sec:approxsol}

The new almost-complex structures on $M$ and $X$ now both coincide with the standard one $J_0$ in suitables regions of the Darboux charts. Thus, we can glue them together to obtain an almost complex structure on the ``connected sum'' manifold $M_\varepsilon$ constructed at the end of paragraph \ref{sec:Darboux}. 

First, we define a function on $M_\varepsilon$ that will encode both the function $\rho$ that extends the distance to the singularities on $M$, and the radius function $r$ on $X$. We set
 \begin{equation*}
   \rho_\varepsilon =
   \begin{dcases}
   \rho \text{ where } \rho \geq 2r_\varepsilon;\\
   \varepsilon h_{\varepsilon^{-1}}^*r  \text{ where } \rho \leq 2r_\varepsilon.
   \end{dcases}
 \end{equation*}

We define $\hat J_\varepsilon$ as follows:
\begin{equation*}
\hat J_{\varepsilon} =
\begin{dcases}
h_{\varepsilon^{-1}}^*J_{R_\varepsilon} & \text{where } \rho_\varepsilon < 2r_\varepsilon,\\
J_{r_\varepsilon} & \text{where  } \rho_\varepsilon \geq 2 r_\varepsilon.
\end{dcases}
\end{equation*}

This smooth section of End($TM_\varepsilon$) defines an almost complex structure on $M_\varepsilon$ that is compatible with $\omega_\varepsilon$ by construction. It is not integrable; its Nijenhuis tensor is supported in a small annulus $\{r_\varepsilon \leq \rho_\varepsilon \leq 4r_\varepsilon\}$ around each singularity.

\begin{lemme}
	The Nijenhuis tensor $N_{\hat J_\varepsilon}$ of $\hat J_\varepsilon$ verifies
	\begin{equation}\label{eq:estimNijenhuis}
  N_{\hat J_\varepsilon} =
  \begin{dcases}
    O(\varepsilon^4r_\varepsilon^{-5}) &\text{ on } \{r_\varepsilon \leq \rho_\varepsilon \leq 2r_\varepsilon\} \\
    O(r_\varepsilon) &\text{ on } \{2r_\varepsilon \leq \rho_\varepsilon \leq 4r_\varepsilon\}.
  \end{dcases}
\end{equation}
Moreover, its derivatives verify
\begin{equation}\label{eq:estimNijenhuisDeriv}
  \del^k N_{\hat J_\varepsilon} =
  \begin{dcases}
    O(\varepsilon^4r_\varepsilon^{-5-k}) &\text{ on } \{r_\varepsilon \leq \rho_\varepsilon \leq 2r_\varepsilon\} \\
    O(1) &\text{ on } \{2r_\varepsilon \leq \rho_\varepsilon \leq 4r_\varepsilon\}.
  \end{dcases}
\end{equation}
\end{lemme}

\begin{proof} To deal with the rescaling, observe that
\begin{align*}
	N_{h_{\varepsilon^{-1}}^*J_{R_\varepsilon}}(X,Y) &= N_{h_{\varepsilon^{-1}}^*J_{R_\varepsilon}}(h_{\varepsilon^{-1}}^*\tilde X, h_{\varepsilon^{-1}}^* \tilde Y)\\
	& = h_{\varepsilon^{-1}}^* N_{J_{R_\varepsilon}}(\tilde X,\tilde Y),
\end{align*}
where $\tilde X$ and $\tilde Y$ can be interpreted as vectors on $X$. Using lemma \ref{lemme:estimNijX}, we thus get the estimate on $\{r_\varepsilon \leq r \leq 2r_\varepsilon\}$. The one on $\{2r_\varepsilon \leq r \leq 4r_\varepsilon\}$ comes directly from lemma \ref{lemme:estimNijM}.

\end{proof}

 \noindent\textit{Remark:} Notice that for the exponent in the second line to be positive (hence for $N_{\hat J_\varepsilon}$ to decrease as $\varepsilon$ becomes small), we need $\beta < \frac45$.\\

\medskip

This construction endows $M_\varepsilon$ with an almost Kähler structure. The suitable Riemannian metric is obtained by setting $\hat g_\varepsilon := \omega(\hat J_\varepsilon \cdot, \cdot)$. Equivalently :
\begin{equation*}
\hat g_{\varepsilon} =
\begin{dcases}
\varepsilon^2 h_{\varepsilon^{-1}}^*g_{R_\varepsilon} & \text{where } \rho_\varepsilon \leq 2r_\varepsilon,\\
 g_{r_\varepsilon} & \text{where } \rho_\varepsilon\geq 2r_\varepsilon.
\end{dcases}
\end{equation*}

\medskip

\section{The equation}\label{sec:Equation}

The goal now is to perturb the almost-Kähler structure on $M_\varepsilon$ into one with constant Hermitian scalar curvature. More precisely, we want to express the resulting equation as a partial differential equation on a function $f$ in a suitable functional space. To do this, we use the construction presented in \ref{sec:hamaction}, to associate a compatible $J_f \in \AC_{\omega_\varepsilon}$ to any $f$.
This would be analogous to the use of the $\del\delbar$-lemma to move the Kähler form $\omega_\varepsilon$ in its cohomology class on a Kähler manifold.

Therefore, the differential operator we are interested in is given by $P: f\mapsto s^\nabla(J_f)$.
More specifically, we want to solve the equation $P(f) =s_{g_M} + \lambda$ for $f$ in a suitable functional space and for some constant $\lambda$.

The strategy is the following. We want to solve this equation using a suitable version of the Inverse Function Theorem.

As a consequence, we write a Taylor development of the operator $P$:

\begin{equation}\label{taylordev}
  s^\nabla(J_f) = s^\nabla(\hat J_\varepsilon) + L_\varepsilon f + Q_\varepsilon(f),
\end{equation}
where $L_\varepsilon$ is the linearisation of the operator at $0$ and $Q_\varepsilon$ contains the nonlinear terms. Thus, we want to solve

\begin{equation}\label{eq:mainequation}
  L_\varepsilon f + \lambda = s_{g_M} -  s^\nabla(\hat J_\varepsilon) -  Q_\varepsilon(f).
\end{equation}

From there, if we can find a right inverse to the operator
\begin{align*}
  \tilde L_\varepsilon : \R \times E &\rightarrow F \\
  (\lambda, f) &\mapsto \lambda + L_\varepsilon f,
\end{align*}
for suitable Banach spaces $E$ and $F$, we are  brought back to a fixed-point problem. To be able to use the fixed point theorem, we need to perform the following steps:
\begin{enumerate}
  \item Introduce weighted Hölder spaces on the connected sum $M_\varepsilon$;
  \item Build a right inverse for $\tilde L_\varepsilon$;
  \item Estimate the nonlinear operator $Q_\varepsilon$;
  \item Estimate the difference between the Hermitian scalar curvature of the approximate solution and the scalar curvature of the orbifold metric $g_M$.
\end{enumerate}

These steps will be the focus of the next sections.

\subsection{Hölder spaces on \texorpdfstring{$M_\varepsilon$}{the connected sum}.} \label{sec:HolSp}

To make our implicit function theorem work, we will need to study elliptic linear differential operators on $M_\varepsilon$, as well as on its ``components'', namely the ALE space $X$ and punctured orbifold $M^* := M\setminus \{p_1 \dots p_k\}$. However $X$ and $M^*$ are noncompact manifolds, and elliptic operators like the Laplacian do not have good properties in ``classical'' Hölder spaces $\mathcal{C}^{k,\alpha}(M^*)$ (resp. $\mathcal{C}^{k,\alpha}(X)$).
\smallskip
As a consequence, we introduce suitable weighted Hölder spaces on $X$, $M^*$ and, from there, on $M_\varepsilon$. We will follow the introduction of such spaces from \cite{BiqRol} (see also \cite{ArePac,Sze3}). For more details on analysis in weighted functional spaces, see for instance \cite{Bar,CBChr,LocMcO}.

\paragraph{On the ALE model.}
Around each $p \in X$, we have a chart mapping the unit ball $B(0,1) \subset \R^4$ to a geodesic ball of radius $\eta r_0$:
\begin{equation*}
	\phi: B(0,1) \rightarrow B(p, \eta r_0).,
\end{equation*}
 where $\eta>0$ is assumed to be very small and $r_0 = r(p)$, where $r$ is a radius function on $X$ defined outside a compact set (for instance the radius of $\CZ$ in an ALE chart at infinity)..

Moreover, thanks to the ALE estimates on the fall-off of the metric, we may assume that
\begin{equation*}
	\phi^*g_X - r_0^2g_0 = \mathcal{O}(r_0^{-4}),
\end{equation*}
and corresponding control on derivatives to order $k$.
\medskip

Then, by definition, a function $f\in\mathcal{C}^{k,\alpha}_{\text{loc}}(X)$ is in $\mathcal{C}^{k,\alpha}_\delta (X)$ if there is a $C >0$ such that,in each such chart,
\begin{equation*}
	\|f \circ \phi \|_{\mathcal{C}^{k,\alpha}} \leq Cr_0^\delta.
\end{equation*}

With this definition, the upshot is that if $\|f\|_{\mathcal{C}^{k,\alpha}_\delta (X)} \leq C$, then $f\in \mathcal{C}^{k,\alpha}(X)$ and, for $i\leq k$,
\begin{align*}
	|\del^i f| \leq c r^{\delta-i},
\end{align*}
where $r$ is the radius function used above.
\smallskip
The weight $\delta$ thus describe the behaviour at infinity of the function $f$.

\medskip
\noindent \textit{Example: } The function $w \mapsto |w|^\gamma$ belongs to $\mathcal{C}^{k,\alpha}_\delta(X)$ if and only if $\gamma \leq \delta$.

\paragraph{On the punctured orbifold.} Recall that we have endowed $M$ with a function $\rho$ that is equal to the distance $p\mapsto d(p,p_i)$ in disjoint neighborhoods of each singularities, and smoothly extended to 1 away from the singularities.
As before, around each $p\in M^*$, we consider maps to a small geodesic ball
\begin{equation*}
\psi : B(0,1) \rightarrow B(p,\eta\, r_0)
\end{equation*}
with $r_0=\rho(p)$ and such that
\begin{equation*}
	\psi^*g_M - r_0^2g_0 = \mathcal{O}_{\mathcal{C}^k}(r_0^2).
\end{equation*}

A function $f\in\mathcal{C}^{k,\alpha}_{\text{loc}}(M^*)$ is in $\mathcal{C}^{k,\alpha}_\delta (M^*)$ if there is a $C >0$ such that, in each such chart,
\begin{equation*}
	\|f \circ \psi \|_{\mathcal{C}^{k,\alpha}} \leq Cr_0^\delta.
\end{equation*}
In this case, $\delta$ keeps track the worse possible behaviour for $f$ near the singularities.

\medskip
\noindent \textit{Example: } The function $z \mapsto |z|^\gamma$ belongs to $\mathcal{C}^{k,\alpha}_\delta(M^*)$ if and only if $\gamma \geq \delta$.

\paragraph{On the connected sum.} 
We define the $\mathcal{C}^{k,\alpha}_{\delta}(M_\varepsilon)$-norm on $M_\varepsilon$ by gluing together the weighted spaces on the two pieces of the gluing. 
 Namely, using a cut-off function $\chi$ that is equal to $1$ outside $\rho_\varepsilon \geq 2r_\varepsilon$ and zero in $\rho_\varepsilon \leq r_\varepsilon$, we can write any tensor field $T$ as the sum of two pieces $T_X := (1-\chi)T $ and $T_{M^*} := \chi T$ respectively supported in $\rho \leq 2r_\varepsilon$ and $\rho \geq 2r_\varepsilon$.
This two pieces thus can be identified to tensor fields on $X$ and $M^*$ respectively. 
Then $\|T\|_{\mathcal{C}^{k,\alpha}_{\delta}}$ is given by
\begin{equation}\label{eq:defWeightNormGluing}
  \varepsilon^{-\ell-\delta}\|(h_{\varepsilon^{-1}})_*T_X\|_{\mathcal{C}^{k,\alpha}_{\delta}(X)}+\|T_{M^*}\|_{\mathcal{C}^{k,\alpha}_{\delta}( M^*)},
\end{equation}
where $\ell$ is the degree of $T$.
This will allow us to decompose the analysis on the ALE and orbifold parts of the gluing, which will prove very useful when constructing a right inverse for the linearised operator.
\ \\
In terms of the `radius' function $\rho_\varepsilon$ on $M_\varepsilon$, the fact that $\|f\|_{\mathcal{C}^{k,\alpha}_{\delta}(M_\varepsilon)} \leq c$ rewrites
\begin{equation*}
	|\del^j f| \leq c\rho_\varepsilon^{\delta-i},
\end{equation*}
for any $j \leq k$; that is to say,
\begin{equation}\label{eq:HoldNormsExplained}
\begin{aligned}
	|\del^j f| &\leq c \text{ where }\rho_\varepsilon\geq 4r_\varepsilon\\
	|\del^j f| &\leq c\rho^{\delta-i}\text{ where } 2r_\varepsilon \leq\rho_\varepsilon\leq 4r_\varepsilon\\
	|\del^j f| &\leq cr^{\delta-i} \text{ where } \rho_\varepsilon\leq 2r_\varepsilon.\\
\end{aligned}
\end{equation}

We have the following relations for the norms with different weights:
\begin{equation*}
  \|f\|_{\mathcal{C}^{k,\alpha}_{\delta'}} \leq
  \begin{dcases}
    \|f\|_{\mathcal{C}^{k,\alpha}_{\delta}} \text{ if } \delta' \leq \delta \\
    \varepsilon^{\delta-\delta'}\|f\|_{\mathcal{C}^{k,\alpha}_{\delta'}} \text{ if } \delta' >\delta.
  \end{dcases}
\end{equation*}
Moreover, note that the multiplication
\begin{align*}
  \mathcal{C}^{k,\alpha}_{\delta} \times \mathcal{C}^{k,\alpha}_{\delta'} &\rightarrow \mathcal{C}^{k,\alpha}_{\delta+\delta'}\\
  (f,g)\mapsto fg
\end{align*}
in continuous, with norm bounded independently of $\varepsilon$.

In terms of these weighted Hölder spaces, we get the following estimate from \eqref{eq:estimNijenhuis} and \eqref{eq:estimNijenhuisDeriv}:
\begin{lemme}
	The Nijenhuis tensor of $\hat J_\varepsilon$ has coefficients in $\mathcal{C}^{3,\alpha}_{0}$ for $0<\alpha<1$, and we have
	\begin{equation*}
		\|N_{\hat J_\varepsilon}\|_{\mathcal{C}^{3,\alpha}_{0}} =
		\begin{dcases}
			O(\varepsilon^4r_\varepsilon^{-5}) &\text{ on } \{r_\varepsilon \leq \rho_\varepsilon \leq 2r_\varepsilon\} \\
    		O(r_\varepsilon) &\text{ on } \{2r_\varepsilon \leq \rho_\varepsilon \leq 4r_\varepsilon\}.
		\end{dcases}
	\end{equation*}
\end{lemme}

\subsection{The linearised operator \texorpdfstring{$L_\varepsilon$}{}.}\label{sec:linop}
The next step is to understand the linearised operator $L_\varepsilon$. We use the computation of the linearised operator performed in Section \ref{sec:hermscalcurv}, Proposition \ref{prop:calcLinearisation} :
\begin{equation*}
		\hat J_\varepsilon\delta( \hat J_\varepsilon\mathcal{L}_{X_f}\hat J_\varepsilon))^\flat = \Delta_{\hat g_\varepsilon} df - 2  \Ric(\grad f, \cdot) + E_\varepsilon f,
	\end{equation*}
	for $f \in C^{3,\alpha}(M_\varepsilon)$, with
	\begin{equation}\label{errorterm}
	\begin{aligned}
		E_\varepsilon f(Y) = &\sum_{i} df((D^2_{e_i, \hat J_\varepsilon Y}\hat J_\varepsilon)e_i) + 2Ddf(e_i, \hat J_\varepsilon (D_Y\hat J_\varepsilon)e_i)
	\end{aligned}
	\end{equation}
	in an orthonormal frame $\{e_1,\dots, e_{2m}\} =\dfrac1{\sqrt{2}}\{Z_1,\dots,Z_m,\hat J_\varepsilon Z_1,\dots,\hat J_\varepsilon Z_m\}$ on $(TM_\varepsilon, \hat g_\varepsilon)$.

\medskip

Thus
\begin{equation}
	L_\varepsilon f =  \Delta_{\hat g_\varepsilon}^2f - 2\delta(\Ric(df)) + \delta E_\varepsilon f,
\end{equation}

As a consequence, the error term in supported in the gluing region $\{r_\varepsilon \leq r \leq 4r_\varepsilon\}$, and we expect it to be small in appropriate weighted Hölder spaces. 
\smallskip

We make this hunch precise in the next lemma.

\begin{lemme}[Estimate on the error term]\label{estimErr}
Let $f\in\mathcal C^{4,\alpha}_\delta(M_\varepsilon)$. Then we have
\begin{equation}\label{eq:NijenhuisWeightNorm}
\|\delta E_\varepsilon f\|_{\mathcal C^{0,\alpha}_{\delta-4}} = o (1) \|f\|_{C^{4,\alpha}_\delta}.
\end{equation}
\end{lemme}

\begin{proof}
Recall that for any vector fields $X, Y$ and $Z$, the following holds:
\begin{equation*}	
	\hat g_\varepsilon((D_X \hat J_\varepsilon)Y, Z) = 2\hat g_\varepsilon(\hat J_\varepsilon X, N(Y,Z);)
\end{equation*}	
thus when computing estimates, the $C^{k,\alpha}_\delta(M_\varepsilon)$-norms of the Nijenhuis tensor and and $D\hat J_\varepsilon$ are comparable.

Applying the codifferential to the error term \eqref{errorterm}, we see that the terms that appear are of the form
\begin{equation}
	\sum_{k=0}^2 \del^k (D\hat J_\varepsilon)\del^{3-k}f
\end{equation}
or
\begin{equation}
	\del^2f (D\hat J_\varepsilon);\ \del^2f (D\hat J_\varepsilon)^2.
\end{equation}

We need to compare these to the $C^{4,\alpha}_\delta$-norm of $f$. Since all these terms are supported in $\{r_\varepsilon \leq \rho_\varepsilon \leq 4r_\varepsilon\}$, by definition of the weighted norms, we have
\begin{align*}
	|\del^j f| &\leq C r_\varepsilon^{\delta-j} \|f\|_{C^{4,\alpha}_\delta},\\
	|\del^k(D\hat J_\varepsilon)| &\leq Cr_\varepsilon^{-k}\|N_{\hat J_\varepsilon}\|_{\mathcal{C}^{3,\alpha}_{0}}
\end{align*}
for some positive constant $C$. Thus we obtain
\begin{itemize}
\renewcommand{\labelitemi}{$\star$}
	\item $|\rho_\varepsilon^{4-\delta}(D\hat J_\varepsilon)\del^{3}f| \leq Cr_\varepsilon  \|N_{\hat J_\varepsilon}\|_{\mathcal{C}^{3,\alpha}_{0}}\|f\|_{C^{4,\alpha}_\delta};$

	\item $|\rho_\varepsilon^{4-\delta}\del(D\hat J_\varepsilon)\del^{2}f| \leq Cr_\varepsilon\|N_{\hat J_\varepsilon}\|_{\mathcal{C}^{3,\alpha}_{0}} \|f\|_{C^{4,\alpha}_\delta};$

	\item $|\rho_\varepsilon^{4-\delta}\del^2(D\hat J_\varepsilon)\del f| \leq Cr_\varepsilon\|N_{\hat J_\varepsilon}\|_{\mathcal{C}^{3,\alpha}_{0}} \|f\|_{C^{4,\alpha}_\delta};$

	\item $|\rho_\varepsilon^{4-\delta}(D\hat J_\varepsilon)\del^{2}f| \leq Cr_\varepsilon^2 \|N_{\hat J_\varepsilon}\|_{\mathcal{C}^{3,\alpha}_{0}} \|f\|_{C^{4,\alpha}_\delta}.$

\end{itemize}
Using \eqref{eq:NijenhuisWeightNorm}, we see that all the right-hand terms are $o(1)$ times $ \|f\|_{C^{4,\alpha}_\delta}$, which is the conclusion we seeked.
\end{proof}

\subsubsection{Mapping properties of the Lichnerowicz operator.}\label{subsec:Lichproperties}

In this section we recall some properties of the "classical" Lichnerowicz operator on the punctured orbifold $(M^*, g_M, J_M)$ and on the ALE space $(X, g_X, J_X)$; those will be used as models to which we shall compare $L_\varepsilon$.

We are especially interested in mapping and Fredholm properties when the operator is defined between weighted spaces. We follow the exposition given in \cite{ArePac}. The analysis can be found in more details in Melrose's book \cite{Mel} (in Sobolev spaces), as well as \cite{PacRiv} (in Hölder spaces).

\paragraph{On the punctured orbifold $M^*$.} The weight allows us to take into account the behavior of functions near the punctures, and it is to be expected that the properties of $\Lich$ will greatly depend on it. More precisely, it turns out that we will need to avoid a discrete set of "bad weights", the \textit{indicial roots.} Roughly, the indicial roots describe the possible behaviors of a function in the kernel of $\Lich$ near the singularity.  Using our chart near each singularity, a real, number $\delta$ is an indicial root if there is a function $v \in \mathcal{C}^\infty(B(p_j, 1))$ such that
\begin{equation*}
  \Lich(\rho^\delta v) = O(\rho^{\delta-3}).
\end{equation*}
Using the fact that, in this chart, the Kähler structure on $M^*$ differs from the Euclidean one at order 2, we see that it is equivalent to look for indicial roots of $\Lambda_0^2$, where $\Lambda_0$ is the Euclidean laplacian. These are known; the computation is recalled in \cite{ArePac} and \cite{Sze2} and rely on the eigenfunctions of the Laplacian on the sphere $S^3$, and are contained in $\Z$.\\

Choosing $\delta$ outside this critical set, we obtain that the operator
\begin{align*}
  \Lich_\delta : C^{4, \alpha}_\delta(M^*) &\rightarrow  C^{0, \alpha}_{\delta-4}(M^*) \\
  f & \mapsto \Lich f
\end{align*}
is well defined, Fredholm, and has closed range. It also verifies the following duality property:
\begin{equation}\label{lichdual}
  \text{dim Ker } \Lich_\delta = \text{dim Coker } \Lich_{-\delta}.
\end{equation}

To obtain good mapping properties, we need to introduce a modification of the operator. For each $i \in \{1,\dots,k\}$, let $\xi_i$ be a smooth function on $M$ supported in a small ball $B(p_j, r_0)$ around $p_j$ and identically equal to 1 in $B(p_j, r_0/2)$. Let $\mathcal{V} = \text{span}(\xi_1,\dots,\xi_\ell)$; we endow $\mathcal V$ with the norm $|f| = \sum|f(p_i)|$. Then we have:

\begin{prop}\label{OrbRightInverse}
  Assume that $\delta \in (0,1)$, $\alpha \in (0,1)$. Then the operator
  \begin{align*}
    \Lich'_\delta : (\mathcal C^{4, \alpha}_\delta \oplus \mathcal{V}) \times \R &\rightarrow \C^{0,\alpha}_{\delta-4} \\
    (f, \nu) &\mapsto \Lich f + \nu
  \end{align*}
  is surjective and has one-dimensional kernel constituted of constant functions.
\end{prop}

A proof of this can be found in \cite{ArePac} (Proposition 5.2).

The Lichnerowicz operator on $(M^*,\omega_M)$ admits a right inverse provided we add a space of functions constant near the singularities at the source. This will come at the cost of a less good norm for the right inverse of $L_\varepsilon$.

\paragraph{On the ALE space $X$.} Most of the previous paragraph applies. This time, an indicial root for $\Lich_\delta$ is characterized by the existence of $v \in \mathcal{C}^\infty(\{r = 1\})$ such that
\begin{equation*}
 \Lich(r^\delta v) = O(r^{\delta - 5}),
\end{equation*}
and indicial roots describe asymptotic behaviors of function in Ker $\Lich$. Due to the decay of the Eguchi-Hanson metric and complex structure towards the Euclidean ones, we may, as before, reduce the problem to seeking indicial roots of $\Delta_0^2$ at infinity. This set is again contained in $\Z$, and, for any $\delta$ outside the critical set,the operator
\begin{align*}
  \Lich_\delta : C^{4, \alpha}_\delta(X) &\rightarrow  C^{0, \alpha}_{\delta-4}(X) \\
  f & \mapsto \Lich f
\end{align*}
is well defined, Fredholm and has closed range. Moreover, the duality property \eqref{lichdual} still holds.

Since there cannot be a holomorphic vector field on $X$ decaying at infinity, observe that for $\delta < 0$, there is no nontrivial solution of $\Lich f = 0$ such that $\phi \in \mathcal{C}^{4,\alpha}_\delta(X).$

As a consequence, we have
\begin{prop}\label{ALERightInverse}
  Assume that $\delta \in (0,1)$. Then $\Lich_\delta$ is surjective and its kernel is of dimension 1, generated by 1.
\end{prop}

Again this proposition is proved in \cite{ArePac}.

\subsubsection{Construction of a right inverse for \texorpdfstring{$L_\varepsilon$}{the linearised operator.}}\label{subsec:rightinverse}
We are now able to build a right inverse for the operator $\tilde L_\varepsilon$. To do this, we will glue together right inverses of $\Lich$ on $M^*$ and $X$, thus obtaining an "approximate right inverse", from which we can build a proper right inverse to $L_\varepsilon$. This proof is the same as in \cite{Sze3}, with the necessary adaptations due to our choice of weights as in \cite{BiqRol}, and the presence of an error term. Factoring this in, we prove

\begin{prop}\label{prop:InvLinOp}
  For a sufficiently small gluing parameter $\varepsilon >0$, the operator
  \begin{align*}
  \tilde L_\varepsilon : \mathcal C^{4,\alpha}_\delta(M_\varepsilon)\times \R &\rightarrow \mathcal C^{0,\alpha}_{\delta-4}(M_\varepsilon) \\
  (f,\nu) &\mapsto L_\varepsilon f +\nu
  \end{align*}
  admits a right inverse $G_\varepsilon$, with operator norm bounded by $\varepsilon^{-\delta\beta^+}$, where $\beta<\beta^+<1$. 
\end{prop}

\begin{proof}
  This proof follows that of Proposition 20 in \cite{Sze3}, which we recall in details here for the sake of completeness. The idea, explained for instance in \cite{DonKro}, is to glue together right inverses on the model spaces, that have been obtained in section \ref{subsec:Lichproperties}, to obtain an approximate right inverse to $\tilde L_\varepsilon$ on the connected sum $M_\varepsilon$. Then, we will modify this approximate right inverse to get a proper right inverse for $\tilde L_\varepsilon$.

 \medskip

We will need two sets of cutoff functions to build the approximate inverse operator. First, let $\gamma :\R \rightarrow [0,1]$ be a smooth function, equal to 0 on $]-\infty, 1]$ and equal to 1 in $[4, +\infty[$. On $M_\epsilon$ we define
\begin{equation*}
\gamma_1 : x\in M_\varepsilon \mapsto \gamma \left(\frac{\rho_\varepsilon(x)}{r_\varepsilon} \right).
\end{equation*}
Then $\gamma_1$ is supported in the region $\rho_\varepsilon \geq r_\varepsilon$, which can be identified with a region of the (punctured) orbifold $M^*$. Its derivative $\del\gamma_1$ is supported in the gluing region $r_\varepsilon \leq \rho_\varepsilon \leq 4r_\varepsilon$.

\medskip
We also set $\gamma_2 := 1-\gamma_1$, supported in $\rho_\varepsilon \leq 4r_\varepsilon$ which can be identified with $4R_\varepsilon \geq r$ in the ALE space $X$.

\medskip
Both $\gamma_1$ and $\gamma_2$ are smooth on $M_\varepsilon$ and are bounded in weighted Hölder norm:
\begin{equation}\label{gammaEstim}
\| \gamma_i \|_{\mathcal C^{4,\alpha}_0} \leq c.
\end{equation}

We will need two other cutoff functions $\zeta_1$ and $\zeta_2$ with a slightly larger support, and with $\zeta_i =1$ in the support of $\gamma_i$. To do this, recall that $r_\varepsilon = \varepsilon^\beta$ with $0<\beta<1$. We choose a slightly larger exponent $\beta^+$ and a slightly smaller exponent $\beta^-$ so that
$0<\beta^-<\beta<\beta^+<1$. Thus the region $\varepsilon < \rho_\varepsilon < 1$ where we perform the gluing is sliced up in regions $1 > 4\varepsilon^{\beta-} > 4r_\varepsilon > 2r_\varepsilon > r_\varepsilon > \varepsilon^{\beta^+} > \varepsilon$.

\medskip

Let now $\zeta^+: \R \rightarrow [0,1]$ be a smooth function such that $\zeta^+(t) = 1$ when $t \leq \beta$, $0$ when $t\geq \beta^+$. The smooth cutoff $\zeta_1$, defined by
\begin{equation*}
\zeta_1: x \in M_\varepsilon \mapsto \zeta^+\left( \frac{\log(\rho(x))}{\log(\varepsilon)} \right),
\end{equation*}
is supported in $\rho \geq \varepsilon^{\beta^+}$ and is equal to 1 in supp $\gamma_1$.

\medskip

Similarly, let $\zeta^-: \R \rightarrow [0,1]$ be a smooth function equal to 1 on $]\beta, +\infty[$ and zero on $]-\infty, \beta^-[$ and define a cutoff on $M_\varepsilon$ by
\begin{equation*}
\zeta_2: x \in M_\varepsilon \mapsto \zeta^-\left( \frac{\log(\rho(x)/4)}{\log(\varepsilon)} \right).
\end{equation*}
Then $\zeta_2$ is supported in $\rho \leq 4\varepsilon^{\beta^-}$ and is equal to 1 in supp $\gamma_2$.

As far as estimations in Hölder norms are concerned, we see that
\begin{equation}\label{estimZeta}
 \|\del\zeta_i\|_{\mathcal C^{3,\alpha}_{-1}} \leq \frac c {|\log\ \varepsilon|}.
\end{equation}

Now let $\psi \in  \mathcal{C}^{0, \alpha}_{\delta-4}$. Notice that $\gamma_1\psi$ can be considered as a function on the punctured orbifold $M^*$. Moreover, using \eqref{gammaEstim}, we have
\begin{equation*}
	\|\gamma_1\psi\|_{ \mathcal{C}^{0, \alpha}_{\delta-4}(M^*)} \leq c \|\psi\|_{ \mathcal{C}^{0, \alpha}_{\delta-4}}.
\end{equation*}
From Proposition \ref{OrbRightInverse}, there is a function $G_1(\gamma_1\psi) = \tilde G_1(\gamma_1\psi)+ \sum \lambda_i \xi_i \in \C^{0,\alpha}_{\delta-4}(M^*) \oplus \mathcal{V}$ and a constant $\nu$ given by
\begin{equation*}
  \nu = \dfrac 1{\vol(M^*)}\int_{M^*} \gamma_1 \psi \vol_{g_M}
\end{equation*}
such that
\begin{equation}\label{estimG1}
	\|\tilde G_1(\gamma_1 \psi)\|_{\mathcal{C}^{4, \alpha}_\delta} + \sum |\lambda_i| + |\nu| \leq c \|\gamma_1\psi\|_{ \mathcal{C}^{0, \alpha}_{\delta-4}(M^*)},
\end{equation}
and
\begin{equation}
 \Lich_M(G_1(\gamma_1\psi))+ \nu = \gamma_1 \psi.
\end{equation}

\medskip
On the other hand, we may consider $\gamma_2 \psi$ as a $\mathcal C^{0,\alpha}_{\delta-4}$ function on $X$. Taking into account the rescaling, we have that
\begin{equation*}
	\|\gamma_2\psi\|_{ \mathcal{C}^{0, \alpha}_{\delta-4}(X)} \leq c \varepsilon^{\delta-4}\|\psi\|_{ \mathcal{C}^{0, \alpha}_{\delta-4}}.
\end{equation*}

Then from Proposition \ref{ALERightInverse} we see that there is a $G_2(\gamma_2 \psi)$ such that
\begin{equation*}
	\| G_2(\gamma_2 \psi)\|_{\mathcal{C}^{4, \alpha}_\delta}(X) \leq c \| \varepsilon^4 \gamma_2\psi\|_{ \mathcal{C}^{0, \alpha}_{\delta-4}(X)} \leq c \varepsilon^\delta\|\psi\|_{ \mathcal{C}^{0, \alpha}_{\delta-4}},
\end{equation*}
thus
\begin{equation}\label{estimG2}
 \| G_2(\gamma_2 \psi)\|_{\mathcal{C}^{4, \alpha}_\delta} \leq c\|\psi\|_{ \mathcal{C}^{0, \alpha}_{\delta-4}},
\end{equation}
and such that
\begin{equation*}
\Lich_X G_2(\gamma_2 \psi) = \varepsilon^4 \gamma_2\psi,
\end{equation*}
thus, after rescaling,
\begin{equation}
\Lich_{\varepsilon^2X} G_2(\gamma_2 \psi) = \gamma_2\psi.
\end{equation}

\medskip
Now we glue these pieces together to get an approximate right inverse for $\tilde L_\varepsilon$. More precisely we set
\begin{equation*}
 \tilde G \psi = \zeta_1 G_1(\gamma_1\psi) +  \zeta_2 G_2(\gamma_2\psi)
\end{equation*}
and we want to show that
\begin{equation*}
	\psi \in \mathcal{C}^{0, \alpha}_{\delta-4} \mapsto (\tilde G \psi, \nu)
\end{equation*}
is an approximate right inverse to $\tilde L_\varepsilon$, and that the operator norm of
\begin{equation}
 \tilde G : \mathcal{C}^{0, \alpha}_{\delta-4} \rightarrow \mathcal{C}^{4, \alpha}_\delta
\end{equation}
is bounded by $\varepsilon^{-\delta\beta^+}$.
\\

We tackle the operator norm first. For $\psi \in  \mathcal{C}^{0, \alpha}_{\delta-4}$ we want to show that
\begin{equation*}
	\| \zeta_1G_1(\gamma_1 \psi) +  \zeta_2 G_2(\gamma_2\psi) \|_{ \mathcal{C}^{4, \alpha}_\delta} \leq \| \zeta_1G_1(\gamma_1 \psi)\|_{ \mathcal{C}^{4, \alpha}_\delta}  + \| \zeta_2 G_2(\gamma_2\psi) \|_{ \mathcal{C}^{4, \alpha}_\delta} \leq C \varepsilon^{\delta \beta^+} \| \psi \|_{\mathcal{C}^{0, \alpha}_{\delta-4}}.
\end{equation*}
The term $\zeta_2 G_2(\gamma_2\psi)$, which can be considered on the ALE space $X$, will not be an issue. Indeed, its norm will be sum of terms of the form
\begin{equation}\label{zetaGdecomp}
	\sum_{j=0}^\ell \rho^j\,|\del^j\zeta_2|\,\rho^{\ell-j-\delta}\,|\del^{\ell-j}( G_2(\gamma_2\psi))|,
\end{equation}
for $\ell=0,\dots, 4$.

 Using \eqref{estimG2} and \eqref{estimZeta}, in addition to the fact that $\zeta_2$ is a bounded function on $M_\varepsilon$, we see that those terms behave at worse like $\mathcal O(\| \psi \|_{\mathcal{C}^{0, \alpha}_{\delta-4}})$.

 The bad estimate comes from the 'orbifold' term  $\zeta_1G_1(\gamma_1 \psi)$. Indeed, $G_1(\gamma_1 \psi)$ is the sum of a $\mathcal C^{4, \alpha}_\delta$ function, to which we may apply the same reasoning as the other term, and a function in $\mathcal{V}$, which behave like a constant near each puncture $p_i$ in $M^*$. Such constants are not bounded in  $\mathcal C^{4, \alpha}_\delta(M^*)$-norm for a positive $\delta$, as is the case here.
 However what we are interested in is $\zeta_1G_1(\gamma_1 \psi)$, with $\zeta_1$ supported in $\{\rho \geq \varepsilon^{\beta^+}\}$, thus we in fact stay at a `safe distance' from the punctures, and the norm of the constants is then comparable to
\begin{equation*}
 	\sup_{\rho \geq \varepsilon^{\beta^+}}\ \lambda_i|\rho^{-\delta}| \leq c\varepsilon^{\delta\beta^+}.
 \end{equation*}

 Thus, using \eqref{estimG1}, in the $\mathcal C^{4, \alpha}_\delta$-norm on $M_\varepsilon$ we get
 \begin{equation*}
 	\|\zeta_1G_1(\gamma_1 \psi)\|_{\mathcal C^{4, \alpha}_\delta} \leq c \varepsilon^{\delta\beta^+} \| \psi \|_{\mathcal{C}^{0, \alpha}_{\delta-4}}.
 \end{equation*}

\medskip
To show that $\tilde G$ does constitute an approximate inverse to $L_\varepsilon$, still following the proof in \cite{Sze3}, we prove the following claim:
\begin{equation}\label{approxRightInv}
 \|L_\varepsilon( \tilde G\psi) + \nu - \psi \|_{\mathcal{C}^{0, \alpha}_{\delta-4}} \leq \frac12 \|\psi\|_{\mathcal{C}^{0, \alpha}_{\delta-4}}.
\end{equation}

To do this, we will separate the study on the different "pieces" of the connected sum and compare with the model operators on $X$ and $M^*$. We write
\begin{equation}\label{decompG}
\begin{aligned}
	L_\varepsilon( \tilde G\psi) + \nu - \psi &= L_\varepsilon(\zeta_1 G_1(\gamma_1\psi)) + \nu - \gamma_1\psi \\
	&+ L_\varepsilon(\zeta_2 G_2(\gamma_2\psi))- \gamma_2\psi.
\end{aligned}
\end{equation}
\medskip
First we deal with the terms on the first line, which live in $\{ \rho_\varepsilon \geq \varepsilon^{\beta^+} \}$. In this region, which can be considered as a subset of $M^*$, we want to compare $L_\varepsilon$ with the model operator $\Lich_M$. We will  need the following lemma:

\begin{lemme}\label{estimgMgeps}
On the region $\{ \rho_\varepsilon \geq \varepsilon^{\beta^+} \}$ in $M_\varepsilon$, the metric $\hat g_\varepsilon$ compares to the orbifold metric $g_M$ as follows:
\begin{equation}
\|\hat g_\varepsilon - g_M \|_{\mathcal{C}^{3, \alpha}_0} = \mathcal O (r_\varepsilon^2 + \varepsilon^{4(1-\beta^+)})
\end{equation}
\end{lemme}

\begin{proof}
We decompose the study of $\hat g_\varepsilon - g_M$ in three regions of $M_\varepsilon$.
\medskip
\begin{itemize}
\item On $\{\rho \geq 4r_\varepsilon\}$, $\hat g_\varepsilon - g_M = 0$ by definition.
\item On $\{2r_\varepsilon \leq \rho \leq 4r_\varepsilon\}$, we have $\hat g_\varepsilon - g_M = \omega_M(J_{r_\varepsilon} - J_M)\cdot, \cdot)$. Using \eqref{JMJ0} we see that $$\|J_{r_\varepsilon}-J_M\|_{\mathcal{C}^{3, \alpha}_0} \leq c r_\varepsilon^2$$.
\item Finally, on the region $\{\varepsilon^{\beta^+} \leq \rho_\varepsilon \leq 2r_\varepsilon\}$, we split in $\hat g_\varepsilon - g_M= \hat g_\varepsilon - g_0 + g_0-g_M$.
Using \eqref{JMJ0} again, we have that $\|g_0-g_M\|_{\mathcal{C}^{3, \alpha}_0} = \|J_0 - J_M\|_{\mathcal{C}^{3, \alpha}_0} = \mathcal O(r_\varepsilon^2)$.

To estimate $\hat g_\varepsilon - g_0$ we identify $\{\varepsilon^{\beta^+} \leq \rho_\varepsilon \leq 2r_\varepsilon\}$ with the region $\{\varepsilon^{\beta^+-1} \leq r \leq 2R_\varepsilon\}$ in $X$.
There, $\hat g_\varepsilon = \varepsilon^2 h^*_{\varepsilon^{-1}}g_{R_\varepsilon}$, thus our ALE estimate \eqref{eq:estimALE} gives $\|\hat g_\varepsilon - g_0\|_{\mathcal{C}^{3, \alpha}_0} = \mathcal O (\varepsilon^{4(1-\beta^+)})$.
\end{itemize}
\end{proof}

Now, using the same reasoning as in Proposition 18 in \cite{Sze3}, we may estimate the operator norm of $L_\varepsilon - \Lich_M$.
Recall that $$\Lich_M f = -\Delta_M^2 f + 2 \delta(\Ric_{g_M}(\grad f, \cdot)),$$and we have obtained earlier that $$L_\varepsilon f= -\Delta_\varepsilon^2 f + 2 \delta(\Ric_{\hat g_\varepsilon}(\grad f, \cdot)) + E(f).$$
Since we are not working in normal holomorphic coordinates, we have to be slightly more careful when comparing the bilaplacians $\Delta^2_M$ and $\Delta_\varepsilon^2$; indeed, the coefficients of the Laplacian $\Delta_M$ in our charts are comparable to $\del \left( g_M^{-1}\del f\right)$, and similarly those of $\Delta_\varepsilon$ are of the form  $\del \left( \hat g_\varepsilon^{-1}\del f\right)$. In particular, notice that first derivatives of the coefficients of the metric intervene.

The coefficients of $\Delta_M^2 f$ are of the form $\del g_M^{-1}\del^2(g_M^{-1}\del f)$, and that of $\Delta_\varepsilon^2 f$ are f the form $\del \hat g_\varepsilon^{-1}\del^2(\hat g_\varepsilon^{-1}\del f)$, thus
\begin{equation*}
\Delta_M^2 f-\Delta_\varepsilon^2 f = \del((g_M^{-1}-\hat g_\varepsilon^{-1})\del^2(g_M^{-1}\del f) + \del (\hat g_\varepsilon^{-1}\del^2( (g_M^{-1}-\hat g_\varepsilon^{-1})\del f).
\end{equation*}
thus
\begin{align*}
\|\Delta_M^2 f-\Delta_\varepsilon^2 f\|_{\mathcal{C}^{0, \alpha}_{\delta-4}} &\leq \|\hat g_\varepsilon - g_M \|_{\mathcal{C}^{3, \alpha}_0} \|\del^2f\|_{\mathcal{C}^{2, \alpha}_{\delta-2}}\\
&\leq \|\hat g_\varepsilon - g_M \|_{\mathcal{C}^{3, \alpha}_0} \|f\|_{\mathcal{C}^{4, \alpha}_{\delta}}.
\end{align*}

On the other hand, in a similar notation, the Riemannian curvature tensor is given by the derivatives of the Christoffel symbols $\Gamma = g^{-1}\del g$, thus
\begin{equation*}
\|\text{Riem}(g_M) - \text{Riem}(\hat g_\varepsilon)\|_{\mathcal C^{0, \alpha}_{-2}} \leq c\|\hat g_\varepsilon - g_M \|_{\mathcal{C}^{2, \alpha}_0}.
\end{equation*}
As a consequence, from Lemmas \ref{estimgMgeps} and \ref{estimErr}, we see that in operator norm, on $\{\rho_\varepsilon \geq \varepsilon^{\beta^+}\}$,
\begin{equation*}
\| L_\varepsilon - \Lich_M \| = o(1).
\end{equation*}

\medskip
In a similar way, we deal with the terms on the second line of \eqref{decompG}, which live in $\{\rho_\varepsilon \leq 4\varepsilon^{\beta^-}\}$. This annulus can be identified with $\{r \leq 4\varepsilon^{\beta^--1}\}$ in $X$. We compare $\hat g_\varepsilon$ with the model ALE metric $g_X$.

\begin{lemme}\label{estimgSgeps}
On the region $\{ \rho_\varepsilon \leq 4\varepsilon^{\beta^-} \}$ in $M_\varepsilon$, the metric $\hat g_\varepsilon$ compares to the rescaled ALE metric $\varepsilon^2 h_{\varepsilon^{-1}}^*g_X$ as follows:
\begin{equation}
\|\hat g_\varepsilon -\varepsilon^2 h_{\varepsilon^{-1}}^*g_X \|_{\mathcal{C}^{3, \alpha}_0} = \mathcal O (\varepsilon^{4}r_\varepsilon^{-4} + \varepsilon^{2\beta^-})
\end{equation}
\end{lemme}
\begin{proof}
As before we split the study between the different parts of $M_\varepsilon$.
\begin{itemize}
\item On $\{\rho_\varepsilon \leq r_\varepsilon\}$, $\hat g_\varepsilon$ is equal to the rescaled ALE metric.
\item On $\{r_\varepsilon \leq \rho_\varepsilon \leq 2r_\varepsilon\}$, $\hat g_\varepsilon -\varepsilon^2 h_{\varepsilon^{-1}}^*g_X = \varepsilon^2\omega_X(J_{R_\varepsilon}-J_X)\cdot, \cdot)$. Using the estimate \eqref{comReps}, we see that on this annulus, $\|\hat g_\varepsilon -\varepsilon^2 h_{\varepsilon^{-1}}^*g_X \|_{\mathcal{C}^{3, \alpha}_0} = \mathcal O (\varepsilon^4r_\varepsilon^{-4})$.
\item Finally, on $\{2r_\varepsilon \leq \rho_\varepsilon \leq 4\varepsilon^{\beta^-} \}$ we write $\hat g_\varepsilon -\varepsilon^2 h_{\varepsilon^{-1}}^*g_X = \hat g_\varepsilon - g_0 +g_0-\varepsilon^2 h_{\varepsilon^{-1}}^*g_X$.
From \eqref{JMJ0} we see that on this region, $\|\hat g_\varepsilon - g_0\|_{\mathcal{C}^{3, \alpha}_0} = \mathcal O (\varepsilon^{2\beta^-})$,
 while the ALE estimate in $\{2R_\varepsilon \leq \rho_X \leq 4\varepsilon^{\beta^- -1}\}$ gives $\|g_0-\varepsilon^2 h_{\varepsilon^{-1}}^*g_X\|_{\mathcal{C}^{3, \alpha}_0} = \mathcal O (\varepsilon^4r_\varepsilon^{-4}).$
\end{itemize}
\end{proof}

From there, the same proof as before shows that in operator norm
\begin{equation*}
\|\Lich_X - L_\varepsilon \| = o(1).
\end{equation*}

Thus, to prove \eqref{approxRightInv}, it is sufficient to show that for $\varepsilon$ small enough, we have
\begin{equation*}
 \|\Lich_M(\zeta_1 G_1(\gamma_1\psi)) + \nu - \gamma_1\psi\|_{\mathcal{C}^{0, \alpha}_{\delta-4}} \leq \frac 14 \|\psi\|_{\mathcal{C}^{0, \alpha}_{\delta-4}}
\end{equation*}
as well as
\begin{equation*}
 \|\Lich_X(\zeta_2 G_2(\gamma_2\psi)) - \gamma_2\psi\|_{\mathcal{C}^{0, \alpha}_{\delta-4}} \leq \frac 14 \|\psi\|_{\mathcal{C}^{0, \alpha}_{\delta-4}}
\end{equation*}

For the first inequality, we have
\begin{align*}
	\Lich_M(\zeta_1 G_1(\gamma_1\psi)) + \nu - \gamma_1\psi &= \zeta_1\Lich_MG_1\gamma_1\psi + A(\grad\zeta_1 \star G_1\gamma_1 \psi) + \nu -  \gamma_1\psi\\
	&=A(\grad\zeta_1 \star G_1\gamma_1 \psi)
\end{align*}
where $A$ is a third-order operator, whose coefficients are bounded in $\mathcal{C}^{0, \alpha}_{\delta-4}$, and $\star$ denotes a bilinear pairing. In fact, the terms contained in $A$ are similar to those appearing in \eqref{zetaGdecomp}.

Thus
\begin{align*}
\|\Lich_M(\zeta_1 G_1(\gamma_1\psi)) + \nu - \gamma_1\psi\|_{\mathcal{C}^{0, \alpha}_{\delta-4}} &= \|A(\grad\zeta_1 \star G_1\gamma_1 \psi)\|_{\mathcal{C}^{0, \alpha}_{\delta-4}}\\
&\leq c \|\del\zeta_1\|_{\mathcal{C}^{3, \alpha}_{-1}}\|G_1\gamma_1 \psi\|_{\mathcal{C}^{3, \alpha}_{\delta}}\\
&=o(1)\|\psi\|_{\mathcal{C}^{0, \alpha}_{\delta-4}}.
\end{align*}

The proof of the second inequality follows broadly the same lines. We have proven \eqref{approxRightInv}, i.e., we have shown that the operator norm of $\tilde L_\varepsilon \circ \tilde G - I $ is less than $1/2$. Thus, $\tilde L_\varepsilon \circ \tilde G$ is invertible and $\tilde{G} \circ (\tilde L_\varepsilon \circ \tilde G)^{-1}$ is a proper right inverse to $\tilde L_\varepsilon$.
\end{proof}

\subsection{Estimation of the Hermitian scalar curvature of \texorpdfstring{$\hat J_\varepsilon$}{the approximate solution}.}
We want to measure how good our approximate solution is in terms of Hermitian scalar curvature, i.e. we want to compare $s^\nabla (\hat J_\varepsilon)$ to the constant scalar curvature on the orbifold $M$. We obtain
\begin{prop}
  Denote by $s_{g_M}$ the constant scalar curvature of $(M, g_M)$. Then, for $0<\delta <1$ and $\beta < \frac23$, we have
  \begin{equation}
    \|s^\nabla(\hat J_\varepsilon) - s_{g_M} \|_{\mathcal{C}^{0,\alpha}_{\delta-4}} = O(\varepsilon^{\beta(4-\delta)}).
  \end{equation}
\end{prop}

\begin{proof}[Proof]
  First recall that
  \begin{equation*}
    s^\nabla(\hat J_\varepsilon) = s_{\hat g_\varepsilon} + |D\hat J_\varepsilon|^2,
  \end{equation*}
  where $D$ is the Levi-Civita connection associated to $\hat g_\varepsilon$. As we already used earlier, $D\hat J_\varepsilon$ has norm comparable to the Nijenhuis tensor, hence
  \begin{equation*}
    |D\hat J_\varepsilon|^2 =
    \begin{dcases}
      O(r_\varepsilon^2) \text{ in } \{2r_\varepsilon \leq \rho \leq 4r_\varepsilon\}\\
      O(\varepsilon^8r_\varepsilon^{-10})\text{ in } \{2r_\varepsilon \leq \rho \leq 4r_\varepsilon\}.
    \end{dcases}
  \end{equation*}
  This error term will be smaller than what we want, so we only need to compare the riemannian scalar curvatures on $M_\varepsilon$ and $M$. The scalar curvature is a constant where $\rho \geq 4r_\varepsilon$ and is bounded in $\{2r_\varepsilon \leq \rho \leq 4r_\varepsilon\}$, as it is given by second derivatives of the metric $g_{r_\varepsilon}$. On the "ALE" side, the scalar curvature is zero where $\rho \leq r_\varepsilon$, and  is given by second derivatives of $g_{R_\varepsilon}$ in $\{ r_\varepsilon \leq \rho \leq 2r_\varepsilon\}$. Thus, using \eqref{comReps} and factoring in the rescaling, we obtain
  \begin{equation*}
    s_{\hat g_\varepsilon} = O(\varepsilon^4 r_\varepsilon^{-6}) \text{ in } \{ r_\varepsilon \leq \rho \leq 2r_\varepsilon\}.
  \end{equation*}
  To sum up,
  \begin{equation*}
     s_{\hat g_\varepsilon}= O(1) + O(\varepsilon^4 r_\varepsilon^{-6}).
  \end{equation*}
  Thus, using that $\rho = O(r_\varepsilon)$ in the region $\{ r_\varepsilon \leq \rho \leq 4r_\varepsilon\}$,
  \begin{align*}
    \rho^{4-\delta}|s^\nabla(\hat J_\varepsilon) - s_{g_M}| &=  \rho^{4-\delta}|s_{\hat g_\varepsilon} +|D\hat J_\varepsilon|^2 - s(M) |\\
    &= O(\varepsilon^4 r_\varepsilon^{-2-\delta}) + O(r_\varepsilon^{4-\delta}) + O(r_\varepsilon^{6-\delta})+O(\varepsilon^8r_\varepsilon^{-6-\delta})\\
    &= O(\varepsilon^{\beta(4-\delta)}),
  \end{align*}
  as soon as $\beta < \frac23$.
\end{proof}

\subsection{Behavior of the nonlinear part.}\label{subsec:nonlin}

Finally, we need to control the nonlinear part of the equation.
Recall the expansion
\begin{equation*}
s^\nabla( J_f) = s^\nabla(\hat J_\varepsilon) + L_\varepsilon f + Q_\varepsilon(f).
\end{equation*}
We prove the following result, following Lemma 19 in \cite{Sze3}.

\begin{lemme}\label{estimNonlin}
There is a constant $C$ such that
  \begin{equation*}
  \| Q_\varepsilon(f) - Q_\varepsilon(g) \|_{\mathcal{C}^{0,\alpha}_{\delta-4}} \leq C\left(\|f\|_{\mathcal{C}^{4,\alpha}_{2}} + \|g\|_{\mathcal{C}^{4,\alpha}_{2}}\right)
  \|f-g\|_{\mathcal{C}^{4,\alpha}_{\delta}}.
  \end{equation*}
\end{lemme}

\begin{proof}
We may rewrite

\begin{equation*}
Q_\varepsilon(f) - Q_\varepsilon(g) = \int_0^1 d_{\chi_t} Q_\varepsilon(f-g)dt,
\end{equation*}
where $\chi_t := g+t(f-g)$. 
Set $h=f-g$. From the Tayor development \eqref{taylordev}, we see that
\begin{equation*}
  \frac d{ds}_{|s=0} Q_\varepsilon(\chi_t+s(f-g)) = d_{J_{\chi_t}}s^\nabla(J_{\chi_t}\L_{X_h}\hat J_\varepsilon) -d_{\hat J_\varepsilon}s^\nabla(\hat J_\varepsilon\L_{X_h}\hat J_\varepsilon),
\end{equation*}
which we rewrite rewrites
\begin{equation}\label{diffNonLin}
  d_{\chi_t} Q_\varepsilon(f-g) =(d_{J_{\chi_t}}s^\nabla - d_{\hat J_\varepsilon}s^\nabla) (J_{\chi_t}\L_{X_h}\hat J_\varepsilon) + d_{\hat J_\varepsilon}s^\nabla\big( (J_{\chi_t}-\hat J_\varepsilon)\L_{X_h}\hat J_\varepsilon\big).
\end{equation}

Observe next that 
\begin{equation*}
	J_{\chi_t}-\hat J_\varepsilon = \big(\exp(\L_{X_{\chi_t}}\hat J_\varepsilon) -I \big)\hat J_\varepsilon,
\end{equation*}
thus its coefficients are comparable to $\del^2 \chi$.
Similarly, the coefficients of $\L_{X_h}\hat J_\varepsilon$ can be expressed in terms of $\del^2h$. 

To deal with the first term of \eqref{diffNonLin}, observe that due to the regularity of $J\in \AC_{\omega_\varepsilon} \mapsto s^\nabla(J)$, the difference $d_{J_{\chi_t}}s^\nabla - d_{\hat J_\varepsilon}s^\nabla$ is controlled by $J_{\chi_t}-\hat J_\varepsilon$. Thus, the weighted norm 
\begin{align*}
	\|(d_{J_{\chi_t}}s^\nabla - d_{\hat J_\varepsilon}s^\nabla) (J_{\chi_t}\L_{X_h}\hat J_\varepsilon)\|_{\mathcal{C}^{0,\alpha}_{\delta-4}} &\leq c \|J_{\chi_t} - \hat J_\varepsilon\|_{\mathcal{C}^{2,\alpha}_{0}}\|J_{\chi_t}\L_{X_h}\hat J_\varepsilon\|_{\mathcal{C}^{2,\alpha}_{\delta-2}}\\
	&\leq c\|\chi_t\|_{\mathcal{C}^{4,\alpha}_{2}}\|h\|_{\mathcal{C}^{4,\alpha}_{\delta}}\\
	&\leq c(\|f\|_{\mathcal{C}^{4,\alpha}_{2}}+\|g\|_{\mathcal{C}^{4,\alpha}_{2}})\|f-g\|_{\mathcal{C}^{4,\alpha}_{\delta}}.
\end{align*}

On the other hand, our computations in section \ref{sec:linop} show that the operator
\begin{equation*}
	d_{\hat J_\varepsilon} s^\nabla : \mathcal{C}^{2,\alpha}_{\delta-2}(\text {End}(TM_\varepsilon)) \rightarrow \mathcal{C}^{0,\alpha}_{\delta-4}
\end{equation*}
is bounded. Thus,
\begin{align*}
	\| d_{\hat J_\varepsilon}s^\nabla\big( (J_{\chi_t}-\hat J_\varepsilon)\L_{X_h}\hat J_\varepsilon\big) \|_{\mathcal{C}^{0,\alpha}_{\delta-4}} &\leq c\|(J_{\chi_t}-\hat J_\varepsilon)\L_{X_h}\hat J_\varepsilon\big)\|_{\mathcal{C}^{2,\alpha}_{\delta-2}}\\
	&\leq c \|(J_{\chi_t}-\hat J_\varepsilon)\|_{\mathcal{C}^{2,\alpha}_{0}}\|\L_{X_h}\hat J_\varepsilon\big)\|_{\mathcal{C}^{2,\alpha}_{\delta-2}}\\
	&\leq c \|\chi_t\|_{\mathcal{C}^{4,\alpha}_{2}}\|h\|_{\mathcal{C}^{4,\alpha}_{\delta}}\\
	&\leq c(\|f\|_{\mathcal{C}^{4,\alpha}_{2}}+\|g\|_{\mathcal{C}^{4,\alpha}_{2}})\|f-g\|_{\mathcal{C}^{4,\alpha}_{\delta}}.
\end{align*}

\medskip

Summing the two final inequalities, we obtain the desired conclusion. 
\end{proof}

\subsection{The nonlinear equation.}

We now have all the tools we need to solve our original equation. We follow closely the proof of Corollary 35 in \cite{BiqRol}. Recall that we seek $f$ and $\lambda$ such that
\begin{equation*}
  L_\varepsilon f + \lambda = s_{g_M} -  s^\nabla(\hat J_\varepsilon) -  Q_\varepsilon(f).
\end{equation*}

We look for $(f, \lambda)$ under the form $G_\varepsilon(\psi)$. Thus this rewrites
\begin{equation}\label{eq:fixedpoint}
\psi =  s_{g_M} -  s^\nabla(\hat J_\varepsilon) -  Q_\varepsilon(G_\varepsilon (\psi)) :=B_\varepsilon(\psi).
\end{equation}
Thus our problem is reduced to a fixed point problem.

\begin{prop}\label{prop:Lipschitz}
	There is a positive constant $C>0$ such that $B_\varepsilon$ maps the ball $\{\|\psi\|_{\mathcal{C}^{0,\alpha}_{\delta-4}} \leq C\varepsilon^2\}$ into itself and is $\frac12$-Lipschitz on this ball.
\end{prop}

\begin{proof}

We have
\begin{equation*}
B_\varepsilon(\psi) - B_\varepsilon(\varphi) =  Q_\varepsilon(G_\varepsilon (\psi)) - Q_\varepsilon(G_\varepsilon (\varphi)).
\end{equation*}
Using Lemma \ref{estimNonlin}, there is a $C_1>0$ such that:
\begin{equation*}
\|  Q_\varepsilon(G_\varepsilon (\psi)) - Q_\varepsilon(G_\varepsilon (\varphi)) \|_{\mathcal{C}^{0,\alpha}_{\delta-4}} \leq C_1\left( \|G_\varepsilon(\psi)\|_{\mathcal{C}^{4,\alpha}_{2}} + \|G_\varepsilon(\varphi)\|_{\mathcal{C}^{4,\alpha}_{2}}\right)
  \|G_\varepsilon(\psi-\varphi)\|_{\mathcal{C}^{4,\alpha}_{\delta}}.
\end{equation*}
Now, $ \|G_\varepsilon(\psi-\varphi)\|_{\mathcal{C}^{4,\alpha}_{\delta}} \leq C_2\varepsilon^{-\delta \beta^+}\|\psi-\varphi\|_{\mathcal{C}^{0,\alpha}_{\delta-4}}$.
On the other hand, since $\psi$ and $\phi$ are assumed to be in $\{\|\psi\|_{\mathcal{C}^{0,\alpha}_{\delta-4}} \leq C\varepsilon^2\}$, we get that
\begin{equation*}
\|G_\varepsilon(\psi)\|_{\mathcal{C}^{4,\alpha}_{\delta}} \leq C_2\varepsilon^{-\delta \beta^+}\|\psi\|_{\mathcal{C}^{0,\alpha}_{\delta-4}} \leq CC_2\varepsilon^{2-\delta \beta^+},
\end{equation*}
and the same stands for $\varphi$.
From this we deduce

\begin{equation*}
\|G_\varepsilon(\psi)\|_{\mathcal{C}^{4,\alpha}_{2}(M_\varepsilon)} \leq C\varepsilon^{\delta-\delta \beta^+}= CC_2 \varepsilon^{\delta(1-\beta^+)}.
\end{equation*}
Thus
\begin{equation*}
\|  Q_\varepsilon(G_\varepsilon (\psi)) - Q_\varepsilon(G_\varepsilon (\varphi)) \|_{\mathcal{C}^{0,\alpha}_{\delta-4}} \leq CC_1C_2 \varepsilon^{\delta(1-2\beta^+)}\|\psi-\varphi\|_{\mathcal{C}^{0,\alpha}_{\delta-4}}.
\end{equation*}
Provided $\beta < \frac12$, this means that for $\varepsilon$ small enough, $B_\varepsilon$ is $\frac12$-contractant on $\{\|\psi\|_{\mathcal{C}^{0,\alpha}_{\delta-4}} \leq C\varepsilon^2\}$.

Moreover, $B_\varepsilon$ maps $\{\|\psi\|_{\mathcal{C}^{0,\alpha}_{\delta-4}} \leq C\varepsilon^2\}$ into itself. Indeed, for such a $\psi$,
\begin{align*}
	\|B_\varepsilon(\psi)\|_{\mathcal{C}^{0,\alpha}_{\delta-4}} &\leq \|B_\varepsilon(\psi)-B_\varepsilon(0)\|_{\mathcal{C}^{0,\alpha}_{\delta-4}} +\|B_\varepsilon(0)\|_{\mathcal{C}^{0,\alpha}_{\delta-4}} \\
	&\leq \frac12 \|\psi\|_{\mathcal{C}^{0,\alpha}_{\delta-4}} +  \|s^\nabla(\hat J_\varepsilon) - \lambda \|_{\mathcal{C}^{0,\alpha}_{\delta-4}} \\
	&\leq \frac12C\varepsilon^2 + C_3\varepsilon^{\beta(4-\delta)}\\
	&\leq C\varepsilon^,
\end{align*}
provided we choose $\beta$ close enough to $\frac23$ and $\delta$ close enough to $0$.
\end{proof}

Thus, we may prove the following result, which directly implies our Theorem \ref{thm:mainresult}.

\begin{thm}\label{thm:conclu}
  For $\varepsilon >0$ small enough, there is on $(M_\varepsilon, \omega_\varepsilon)$ a smooth compatible almost-Kähler structure $J_\varepsilon$, of constant Hermitian scalar curvature, such that 
  \begin{itemize}
    \item $J_\varepsilon$ converges, in $\mathcal C^{2,\alpha}$-norm, to $J_M$, on every compact set of $M^*$ (in the sense of Definition \ref{def:convCompM});
    \item  $J_\varepsilon$ converges, in $\mathcal C^{2,\alpha}$-norm, to $J_X$, on every compact set of $X$ (in the sense of Definition \ref{def:convCompX}).
  \end{itemize}
\end{thm}

\begin{proof}
  According to Proposition \ref{prop:Lipschitz}, we may apply Banach's fixed point theorem to $B_\varepsilon$ on $$\{\|\psi\|_{\mathcal{C}^{0,\alpha}_{\delta-4}} \leq C\varepsilon^2\}.$$ Therefore, there is a unique $\psi_\varepsilon\in \mathcal{C}^{0,\alpha}_{\delta-4}(M_\varepsilon)$, whose norm is comparable to $\varepsilon^2$, and that is solution to the main equation \eqref{eq:fixedpoint}. 

  Then, setting $(f_\varepsilon, \lambda_\varepsilon) = G_\varepsilon(\psi)$, we see that $f_\varepsilon$ solves \eqref{eq:mainequation}, and thus, the almost-complex structure $ J_\varepsilon:=J_{f_\varepsilon}$ endows $M_\varepsilon$ with a constant Hermitian curvature almost-Kähler structure. 
Moreover, by Proposition \ref{prop:InvLinOp}, we have
\begin{equation}\label{eq:estimJfJeps}
  \|J_\varepsilon - \hat J_\varepsilon \|_{\mathcal{C}^{2,\alpha}_{\delta-2}} \leq c \|f_\varepsilon\|_{\mathcal{C}^{4,\alpha}_{\delta}}\leq c \varepsilon^{2- \delta\beta^+}.
\end{equation}
Thus, if $K_1$ is a compact set in $M^*$, then for $\varepsilon$ small enough, $K_1 \subset M\setminus \cup_i B(p_i, 4r_\varepsilon)$. By definition $\hat J_{\varepsilon|K_1} = J_{M|K_1}$. 

Moreover, on $\subset M\setminus \cup_i B(p_i, 4r_\varepsilon)$, the weighted Hölder norm $\mathcal{C}^{2,\alpha}_{\delta-2}$ coincides with the usual Hölder  $\mathcal{C}^{2,\alpha}$ norm (according to the definition \eqref{eq:HoldNormsExplained}), thus \eqref{eq:estimJfJeps} implies
\begin{equation*}
   \|J_\varepsilon -  J_M \|_{\mathcal{C}^{2,\alpha}(K_1)} \leq c \varepsilon^{2- \delta\beta^+}.
\end{equation*}
Since we have chosen $0 < \delta, \beta^+ <1$, we see that the right hand side goes to zero when $\varepsilon$ goes to zero, thus $J_\varepsilon$ does converge to $J_M$ on $K_1$.

\medskip

Similarly, on a compact set $K_2$ of $X$, the pullback $h_\varepsilon^*\hat J_\varepsilon$ is equal to the ALE complex structure $J_{X|K_2}$ for $\varepsilon$ small enough. 

Then, the estimate \eqref{eq:estimJfJeps}, and the definition of  the weighted norms on $M_\varepsilon$ \eqref{eq:defWeightNormGluing} imply that, on $K_2$, we have
 \begin{equation}\label{eq:estimJfJx}
  \|h_\varepsilon^* J_\varepsilon - J_X \|_{\mathcal C^{2, \alpha}(K_2)} \leq c\varepsilon^{\delta(1-\beta^+)}
 \end{equation}
 for some positive constant $c$. Since $\delta \in (0,1)$ and $\beta^+ <1$, the right hand side goes to zero when $\varepsilon$ goes to zero. 

It remains to show that the solution has the required regularity.  The $\mathcal C^{4,\alpha}$ function $f_\varepsilon$ is solution of 
  \begin{equation*}
    s^\nabla(J_{f_\varepsilon}) = \tilde \lambda_\varepsilon,
  \end{equation*}
  with $\tilde  \lambda_\varepsilon$ a constant. As evidenced by the computations of Section \ref{sec:hermscalcurv}, this equation is a 4th order elliptic equation. Moreover, the  coefficients are \emph{rational} functions of $x \in M_\varepsilon$ and derivatives of $f$ up to order 4.

  Using classical results in elliptic regularity (see for instance Besse \cite{Bes}, Theorem 41 in the Appendix, or  Morrey \cite{Morr2}), and a bootstrapping argument, we see that the function $f_\varepsilon$ is actually a smooth function on $M_\varepsilon$.

  As a consequence, the almost-complex structure $J_\varepsilon = J_{f_\varepsilon}$ and the associated metric $g_\varepsilon = \omega_\epsilon(J_\varepsilon\cdot, \cdot)$ are also smooth.
  This concludes the proof of our main result.
\end{proof}

Furthermore, we can refine the bootstrapping argument to obtain that for any $k\geq 0$, the  constant hermitian scalar curvature almost-Kähler structures $(J_\varepsilon)$ converge, in $\mathcal C^{k,\alpha}$ to $J_M$ (resp. $J_X$) on every compact set of $M^* = M \setminus \{p_1, \dots, p_\ell\}$ (resp. on every compact set of $X$).

To obtain this result, we need to show that $\|f_\varepsilon\|_{\mathcal C^{k,\alpha}(K)} \xrightarrow{\varepsilon \rightarrow 0} 0$ for every $k\geq 0$ and for every compact set $K \subset M^*$ (and the same on $X$).
We know that $f_\varepsilon$ is smooth and that the previous convergence holds in $\mathcal C^{4,\alpha}(K)$.

We will make use of the elliptic equation verified by $f_\varepsilon$: there is a constant $\lambda_\varepsilon$ such that
\begin{equation}\label{eq:csthermscalcurv}
  \lambda_\varepsilon = s^\nabla(J_{f_\varepsilon}) = s^\nabla(\hat J_\varepsilon) + L_\varepsilon(f) + Q_\varepsilon(f).
\end{equation}

First, we need the following technical lemma to better understand the non-linear part $Q_\varepsilon$ of the equation.

\begin{lemme}\label{lemma:nonLinDecomp}
  The non-linear part of \eqref{eq:csthermscalcurv} can be decomposed as
  \begin{equation*}
      Q_\varepsilon(u) = Q^{(1)}_\varepsilon(u)+Q^{(2)}_\varepsilon(u),
  \end{equation*} 
  where
  \begin{itemize}
    \item $Q^{(1)}_\varepsilon$ is a nonlinear operator or order 3 with smooth coefficients depending on $\hat g_\varepsilon$ and its derivatives;
    \item $Q^{(2)}_\varepsilon$ is a nonlinear operator of order 4, that verifies, for $u$ suitably smooth and $k \in \N$, 
    \begin{equation*}
      \|Q^{(2)}_\varepsilon(u)\|_{\mathcal C^{k,\alpha}} \leq c \|u\|_{\mathcal C^{k+3}}\|u\|_{\mathcal C^{k+4, \alpha}}
    \end{equation*}
  \end{itemize}
\end{lemme}

\begin{proof}
  To better understand the nonlinearities of the equation, we turn to the computation of the connection 1-form $\alpha$ of the Chern connection. Recall a few notations: the variation of complex structure induced by a function $u$ was given by 
  \begin{equation*}
    a_u = \frac 12  \mathcal{L}_{X_u}\hat J_\varepsilon,
  \end{equation*}
  which is linear in $u$, with derivatives of order at most 2.
  We set
  \begin{equation*}
    J_u = \exp(-a_u)\hat J_\varepsilon \exp(a_u).
  \end{equation*}
  Then in the proof of the Mohsen formula, we had obtained
\begin{align*}
  \alpha(J_u)(X) = &-\frac12\sum_k \hat g_\varepsilon(\exp(a_u)(D_{\exp(-a_u)e_k} J_u)X, e_k)\\
  &+\frac12\sum_k \hat g_\varepsilon(\exp(a_u)D_{J_uX}(\exp(-a_u)e_k), e_k)\\
  &-\frac12\sum_k \hat g_\varepsilon(\hat J_\varepsilon \exp(a_u)D_{X}(\exp(-a_u)e_k), e_k).
\end{align*}

From there, we see that $\alpha_u$ can be written
\begin{equation*}
  \alpha(J_u) = \alpha_0 + \dot \alpha_u + \tilde Q_\varepsilon^{(1)}(u) +  \tilde Q_\varepsilon^{(2)}(u)
\end{equation*}
where $\alpha_0$ is the connection 1-form associated to the approximate solution $\hat J_\varepsilon$, $\dot \alpha_u$ is the linearization. What we are interested in are the remaining terms $\tilde Q_\varepsilon^{(1)}$ and $\tilde Q_\varepsilon^{(2)}$. 
The derivatives of $u$ appearing in  $\tilde Q_\varepsilon^{(1)}(u)$ are of order at most 3. In fact, $\tilde Q_\varepsilon^{(1)}(u)$ is a sum of terms of the form
\begin{equation*}
  (\del u)^l(\del^2u)^k(D_{e_k}\hat J_\varepsilon), (\del u)^l(\del^2u)^k (D_{\hat J_\varepsilon X}e_k)\ \text{ and }(\del u)^l(\del^2u)^k (D_{ X}e_k),
\end{equation*}
with $k+l\geq 2$, and coefficients given by coefficients of the metric $\hat g_\varepsilon$.
On the other hand,  $\tilde Q_\varepsilon^{(2)}(u)$ is a sum of terms of the form
\begin{equation*}
  (\del u)^l(\del^2u)^k \del^3 u
\end{equation*}
for $k+l\geq 1$, and as before the coefficients are provided by that of the metric $\hat g_\varepsilon$.

Since
\begin{equation*}
  s^\nabla(J_u) = 2 \Lambda d \alpha(J_u)
\end{equation*}
we see that 
\begin{equation*}
  s^\nabla(J_u) =  s^\nabla(\hat J_\varepsilon) + L_\varepsilon(u) + Q^{(1)}_\varepsilon(u)+ Q^{(2)}_\varepsilon(u)
\end{equation*}
where the terms in  $Q^{(1)}_\varepsilon$ are of the form $(\del^2u)^k(\del^3 u)$, with $k\geq 0$, and  the terms in $Q^{(2)}_\varepsilon(u)$ are of the form $(\del^2u)^k(\del^4 u)$ for $k\geq 1$.
\end{proof}

Using this, we prove

\begin{prop}
  For all $k \geq 1$, for every compact set $K \subset M^*$, we have $\|f_\varepsilon\|_{\mathcal C^{4+k,\alpha}(K)} \xrightarrow{\varepsilon \rightarrow 0} 0$
\end{prop}

\begin{proof} 
  We prove the claim by induction on $k$. 

  For $k=1$, we want to obtain an estimate on $\|f_\varepsilon\|_{\mathcal C^{5,\alpha}(K)}$. Let $K' \supset K$ a slightly larger compact of $M^*$ and let us consider $\varepsilon$ small enough so that $K' \subset M \setminus \cup_i (B(p_i, 4 r_\varepsilon)$, so that the approximate solution coincides with the orbifold structure on $K'$. For this choice of $\varepsilon$, the smooth function $f_\varepsilon$ is solution of the elliptic fourth-order equation
\begin{equation}\label{eq:csthermscalcurvK}
  \tilde \lambda_\varepsilon  = L(f_\varepsilon) + Q^{(1)}(f_\varepsilon) + Q^{(2)}(f_\varepsilon).
\end{equation}
Here we use that there is an $\varepsilon_{K'}$ such that for $\varepsilon < \varepsilon_{K'}$, on the compact $K'$, the coefficients of the equation do not depend on $\varepsilon$. Moreover, for $\varepsilon < \varepsilon_{K'}$, $s^\nabla(\hat J_\varepsilon)$ is constant, equal to $s_{g_M}$ and $\tilde \lambda_\varepsilon = \lambda_\varepsilon - s_{g_M}$ goes to zero when $\varepsilon$ goes to zero. The equation \eqref{eq:csthermscalcurvK} is quasi-linear, elliptic, of order 4 in $f_\varepsilon$, and its coefficients do not depend on $\varepsilon$.

As a consequence, according to the technical lemma \ref{lemma:nonLinDecomp}, there is some positive constant $c$ such that
\begin{align*}
  \| (L + Q^{(2)})( f_\varepsilon) \|_{\mathcal C^{1, \alpha}(K')} &\leq \| \lambda_\varepsilon \|_{\mathcal C^{1, \alpha}(K')} + \|Q^{(1)}(f_\varepsilon)\|_{\mathcal C^{1, \alpha}(K')} \\
  &\leq  \| \lambda_\varepsilon \|_{\mathcal C^{1, \alpha}(K')} + c\|f_\varepsilon\|_{\mathcal C^{4, \alpha}(K')}.
\end{align*}
Now, according to lemma \ref{lemma:nonLinDecomp}, $(L + Q^{(2)})f_\varepsilon$ is a fourth-order elliptic operator, quasilinear, and the coefficients, which depend on $f_\varepsilon$, are in $\mathcal C^{4, \alpha}(K')$; since $\|f_\varepsilon\|_{\mathcal C^{4,\alpha}(K')}\xrightarrow{\varepsilon\rightarrow 0} 0$, this operator is really a quasilinear perturbation of the linear elliptic operator $L$. 

More precisely, we can rewrite \eqref{eq:csthermscalcurvK} under the form
\begin{equation}\label{eq:cstHermScalDecomp}
  \sum_{|\alpha| = 4} a_\alpha(x, \del f_\varepsilon, \del^2 f_\varepsilon, \del^3 f_\varepsilon)\del^4_\alpha f_\varepsilon = G(x, \del f_\varepsilon, \del^2 f_\varepsilon, \del^3 f_\varepsilon)
\end{equation}
where the operator
\begin{equation*}
   \sum_{|\alpha| = 4} a_\alpha(x, 0, 0, 0)\del^4_\alpha u
\end{equation*}
is linear elliptic. Thus, for $\varepsilon$ small enough, $\|f_\varepsilon\|_{\mathcal{C}^{4, \alpha}(K')}$ is sufficiently small for the left-hand side of \eqref{eq:cstHermScalDecomp} to still be elliptic, with coefficients bounded in $\mathcal C^{1,\alpha}$. Thus, elliptic regularity results (see Morrey \cite{Morr}) imply that
\begin{equation*}
  \| f_\varepsilon \||_{\mathcal C^{5, \alpha}(K)} \leq c_2(\| (L+Q^{(2)})( f_\varepsilon) \|_{\mathcal C^{1, \alpha}(K')} + \| f_\varepsilon \||_{\mathcal C^{0}(K')})
\end{equation*}
Since we know that $\|f_\varepsilon \|_{\mathcal C^{4, \alpha}(K')} \xrightarrow{\varepsilon \rightarrow 0} 0 $, we know, in particular, that $\| f_\varepsilon \||_{\mathcal C^{0}(K)})\xrightarrow{\varepsilon \rightarrow 0} 0$. For $\varepsilon$ small enough, the above estimate rewrites
\begin{align*}
  \| f_\varepsilon \||_{\mathcal C^{5, \alpha}(K)} &\leq c_3( \| \lambda_\varepsilon \|_{\mathcal C^{1, \alpha}(K')} + \|f_\varepsilon\|_{\mathcal C^{4, \alpha}(K')}).
\end{align*}

Thus, we have obtained that on every compact set $K \subset M^*$, $ \| f_\varepsilon \||_{\mathcal C^{5, \alpha}(K)}\xrightarrow{\varepsilon \rightarrow 0} 0$.

\medskip

It remains to show the induction step, which works in the exact same way. Assume, by induction hypothesis, that for every compact set $K' \subset M^*$, $ \| f_\varepsilon \|_{\mathcal C^{k, \alpha}(K')}\xrightarrow{\varepsilon \rightarrow 0} 0$. Let $K \subset M^*$ a compact subset, we want to show that  $ \| f_\varepsilon \||_{\mathcal C^{k+1, \alpha}(K)}\xrightarrow{\varepsilon \rightarrow 0} 0$. Let $K'$ be a slightly bigger compact subset of $M^*$. Choosing $\varepsilon$ small enough, we see that $f_\varepsilon$ is solution of \eqref{eq:csthermscalcurvK} on $K'$. We then go through the same steps to obtain the desired result, in a boostrap-type reasoning. The coefficients of the operator $(L + Q^{(2)})(f_\varepsilon)$ are then in $\mathcal C^{k, \alpha}(K')$ by induction hypothesis, ensuring we may use the elliptic regularity theorem at each step.
\end{proof} 

With the exact same proof, we show

\begin{prop}
  For all $k \geq 1$, for every compact set $K \subset X$, we have $\|h_\varepsilon J_\varepsilon - J_X\|_{\mathcal C^{4+k,\alpha}(K)} \xrightarrow{\varepsilon \rightarrow 0} 0$.
\end{prop}

\section{Hamiltonian stationary spheres.}\label{sec: HamStat}
Through our construction, we have obtained a family of compatible almost-complex structures $( J_f)$ depending on a parameter $0 < \varepsilon < \varepsilon_0$ in such a way that the almost-Kähler structure $(\omega_\varepsilon,  J_\varepsilon,   g_\varepsilon)$ on $M_\varepsilon$ has constant Hermitian scalar curvature for $0 < \varepsilon < \varepsilon_0$.

 Moreover, when $\varepsilon$ goes to zero, the pullback of $J_f$ on the ALE model $X$ converges in $\mathcal C^{2,\alpha}$-norm to $J_X$ in a compact neighborhood of the zero section of $T^*S^2 \simeq X$, in the sense defined in \ref{def:convCompX}, according to Theorem \ref{thm:conclu}. 

\begin{rem} More precisely, in the proof of Theorem \ref{thm:conclu}, we had obtained
 \begin{equation*}
  \|h_\varepsilon^* J_f - J_X \|_{\mathcal C^{2, \alpha}(X)} \leq c\varepsilon^{\delta(1-\beta^+)}
 \end{equation*}
 which also gives us
  \begin{equation}\label{eq:estimgfgx}
 	\|\varepsilon^{-2}h_\varepsilon^* g_f - g_X \|_{\mathcal C^{2, \alpha}(X)} \leq c\varepsilon^{\delta(1-\beta^+)}.
 \end{equation}
\end{rem}
\medskip

Besides, according to Corollary \ref{cor:symplectEquiv}, the symplectic manifolds $(M_\varepsilon, \omega_\varepsilon)$ can actually all be identified to the same symplectic manifold that we call $(\hat M, \hat \omega)$ (for instance by fixing some $\varepsilon_1$). For $\varepsilon \in (0,\varepsilon_0)$, we denote $J_\varepsilon$ the pullback of $J_f$ on $\hat M$ and $g_\varepsilon$ the pullback of $g_f$ on $\hat M$, and $(J_0, g_0)$ the pullback of the approximate solution $(J_\varepsilon, g_\varepsilon)$. Thus, we have a smooth family of almost-Kähler structures $(J_\varepsilon, g_\varepsilon)_{0 \leq \varepsilon < \varepsilon_0}$ on a fixed symplectic manifold $(\hat M, \hat \omega)$.

\medskip

Observe that in the ALE model space $(X=T^*S^2, \omega_X = dd^c u)$, the zero section $S_0$ of $T^*S^2 \rightarrow S^2$ is a Lagrangian sphere. Moreover, $T^*S^2$ is an hyperKähler manifold, and for a different choice of complex structure in the hyperKähler family (namely, the choice that yields the minimal resolution of $\CZ$), the zero section is actually a holomorphic copy of $\C P^1$. 

It is a well-known consequence of Wirtinger's inequality that holomorphic surfaces minimize volume in their homology class. 

The zero section is not holomorphic for our choice of complex structure on $T^*S^2$, but it still is minimal, since we have endowed $T^*S^2$ with the Eguchi-Hanson metric. In particular, it is Hamiltonian stationary, which is to say that it verifies \eqref{eq:defHamStat}.

\medskip

This implies that, when performing the gluing construction in Darboux charts, as we did in section \ref{sec:Darboux}, $S_0$ provides a Hamiltonian stationary (actually, minimal) sphere $S$ in the connected sum manifold $(\hat M, \hat \omega, J_0, g_0)$. 

\medskip

A natural question, therefore, is the following: For positive, small enough, $\varepsilon$, is there a representative of the homology class of $[S]$ - more precisely, a Hamiltonian deformation of $S$ - that is a Hamiltonian stationary sphere for the metric $  g_\varepsilon$ ?

\medskip
We prove that the answer is yes,  extending what has been obtained in \cite{BiqRol} to the case of almost-Kähler smoothings.

\medskip

We need to find representative of the vanishing cycle $[S]$ that verify the equation \eqref{eq:defHamStat} with respect to the metric $g_\varepsilon$, for $\varepsilon$ small enough. It was proven by Oh \cite{Oh1}, Theorem 1, that the corresponding Euler-Lagrange equation is
\begin{equation}\label{eq:HamStatEulLag}
		\delta_{\varepsilon}\alpha_{\varepsilon} = 0,
	\end{equation}
	where $\delta_{\varepsilon}$ is the codifferential associated to the metric $\hat g_{\varepsilon}$, and $\alpha_{\varepsilon}$ is the Maslov form:
	\begin{equation*}
		\alpha_{\varepsilon} := H_{\varepsilon} \lrcorner \hat \omega,
	\end{equation*}
	where $H_{\varepsilon} $ is the mean curvature vector. 

\medskip

Consider the embedding 
	\begin{equation*}
		\iota_0: S^2 \hookrightarrow \hat M
	\end{equation*}
	of the Lagrangian sphere in $(\hat M, \hat \omega)$ that is minimal for $( J_0,  g_0)$. 

	By Weinstein's Lagrangian neighborhood theorem (see \cite{McDSal}, Theorem 3.3), we can identify a neighborhood of $\iota_0(S^2)$ with a neighborhood $\mathcal U$ of the zero section in $(T^*S^2, -d \lambda)$ by a symplectomorphism $\psi$. 
	Hamiltonian deformations of $\iota_0$ are therefore given by functions $u\in \mathcal C^\infty(S^2)$ such that $\|du\|_{\mathcal C^0}$ is small enough that $du \in \mathcal U$. For such a function $u$ we denote
	\begin{equation*}
		i_u:S^2 \hookrightarrow \mathcal U
	\end{equation*}
	the associated immersion. 
	We still denote by $ J_\varepsilon$ and $ g_\varepsilon$ the almost complex structure and associated metrics pulled back by $\psi$ on $\mathcal U$. 
	Let $g_{\varepsilon, u}$ be the restriction of $ g_\varepsilon$ to $\iota_u(S^2)$. Then the immersion $\iota_u$ is Hamiltonian stationary for $ g_\varepsilon$ if it is a critical point for the volume functional
	\begin{equation*}
		u \mapsto \int_{i_u(S^2)} vol_{g_{\varepsilon, u}}.
	\end{equation*}

\medskip

Notice that this equation is not linear in $u$, the induced metric on $S^2$ depends on the embedding encoded by $du$. The linearisation $\mathcal L$ at $0$, in the Kähler setting, is given by Oh's formula (\cite{Oh1}, Theorem 3.4).
	He proves the following on a Kähler manifold: Let $u_t$ be a family of functions on $S^2$, such that $u_0=0$, giving a Hamiltonian deformation $S_t := \iota_{u_t}(S^2)$. Then
	\begin{equation}\label{eq:OhSecondVar}
	\begin{aligned}
		\frac{d^2}{dt^2}_{|t=0} Vol(S_t) &= \int_{S_0} \dot u \mathcal L \dot u\ \vol_0  \\
		&= \int_{S_0} \langle \Delta_0 d\dot u, d\dot u \rangle - \Ric_0( J_0 d\dot u, J_0 d \dot u) - 2 \langle d\dot u \otimes d \dot u \otimes \alpha_0, S \rangle + \langle d\dot u, \alpha_0\rangle^2 \vol_0
	\end{aligned}
	 \end{equation}
	where $\alpha_0$ is the Maslov form for $\iota_0$, $\Ric_0$ is the Ricci curvature of $\hat g_0$ restricted to $S_0=i_0(S^2)$, and $\vol_0$ is the associated volume.
In our setting, the manifold $(\hat M, \hat \omega,  J_0)$ is not Kähler; however,up to reducing the Lagrangian neighborhood, we may assume that the structure $(\hat \omega,  J_0,  g_0)$ is Kähler on $\mathcal U$, since we may thus avoid the region where the Nijenhuis tensor does not vanish. As a consequence, we may apply Oh's formula, as in its proof, the Kähler hypothesis is only used at $t=0$.

\medskip

This allow us to prove:

\begin{prop}
  For  $\varepsilon$ small enough, the almost Kähler manifold $(\hat M, \hat \omega,  J_\varepsilon)$ admits a Lagrangian homology class that is represented by a Hamiltonian stationary sphere.
\end{prop}

\begin{proof}
	
	Consider the operator
	\begin{align*}
		B:  \mathcal C^{2,\alpha}(\AC_{\hat \omega}) \times \mathcal C^{4,\alpha}(S^2) &\rightarrow \mathcal C^{0,\alpha}(S^2) \\
		(J, u) & \mapsto \delta_{J, u}\alpha_{J, u}.
	\end{align*}
	The operator is well defined on the family $(J_\varepsilon)$. Indeed, in local coordinates on $L$, if
	\begin{equation}
		\begin{dcases}
			g_{\varepsilon,u} = g_{ab}dx_a dx_b \\
			\alpha_{\varepsilon, u}= \alpha_a dx_a
		\end{dcases}
	\end{equation}
	then 
	\begin{equation*}
		 \delta_{\varepsilon, u}\alpha_{\varepsilon, u} = - \frac{\del h^{ab}}{\del x_b}\alpha_a - h^{ab}\frac{\del \alpha_a}{\del x_b} - \frac12 h^{ab}\alpha_a \frac{\del}{\del x_b}(\log(\det h_{cd})).
	\end{equation*}
	Thus, the equation invovles first derivatives of the coefficients of $\alpha_{\varepsilon, u}$ and $g_{\varepsilon,u}$. Now, by definition, $g_{\varepsilon,u}$ involves second-order derivatives of $u$, as well as the coefficients of $g_\varepsilon$. The mean curvature vector (thus, the Maslow form) therefore involves third-order derivatives of $u$ and first-order derivatives of the coefficients of $g_\varepsilon$. Finally, as a whole, the equation is of order 4 in $u$ and its coefficients involve second derivatives of $g_\varepsilon$; we conclude using estimates  \eqref{eq:estimJfJx} and \eqref{eq:estimJfJeps}.

\medskip

	It verifies $B(J_X,0)=0$, and, by \eqref{eq:HamStatEulLag}, our problem reduces to finding zeroes of $u \mapsto B(J_\varepsilon, u)$ for $\varepsilon$ small enough. We therefore need to apply the Implicit Function Theorem to $B$ at $(J_X,0)$.

	The linearisation of  $u \mapsto B(J, u)$ at $(J_X,0)$ is given by \eqref{eq:OhSecondVar}. 
	In our framework, $S_0$ is actually minimal, thus $\alpha_0$ vanishes. Moreover, $ g_0$ is given on $\mathcal U$ by the Ricci-flat Eguchi-Hanson metric. Thus in our setting, we get
	\begin{equation*}
		\mathcal L \dot u = \Delta^2 \dot u.
	\end{equation*}

	Thus, since constant functions $u$ result in trivial deformation, we have that for $k>4$, $\mathcal L$ realizes an isomorphism between the Hölder spaces
	\begin{equation*}
		\mathcal L : \mathcal C^{k, \alpha}(S^2)/\R \rightarrow \mathcal C^{k-4, \alpha}_0(S^2) := \left\{f \in \mathcal C^{k-4, \alpha}(S^2),\ \int_{S^2} f \vol_{g_{0,0}} = 0\right\}.
	\end{equation*}
	This observation, along with the estimate \eqref{eq:estimJfJx}, allows us to apply the inverse function theorem to 
	\begin{align*}
		B: \mathcal C^{2,\alpha}(\AC_{\hat \omega}) \times \mathcal C^{4, \alpha}(S^2)/\R  &\rightarrow \mathcal C^{0, \alpha}_0(S^2) \\
		(J, u) & \mapsto \delta_{J, u}\alpha_{J, u}
	\end{align*}
	at $(J_X,0)$; in particular for $\varepsilon$ small enough, there is a unique $u_\varepsilon \in \mathcal C^{4, \alpha}(S^2)/\R$ such that the embedding $\iota_{u_\varepsilon} : S^2 \hookrightarrow \mathcal U$ is Hamiltonian stationary for the metric $g_{\varepsilon}$. 

  Now, $u_\varepsilon$ is solution of the 4th order elliptic equation
\begin{equation*}
  B(J_\varepsilon, u_\varepsilon) = 0.
\end{equation*}
Since, according to theorem \ref{thm:mainresult}, $J_\varepsilon$ is actually smooth, and so is the associated metric whose coefficients appear in the expression of the differential operator $B$, we can, once again, use a bootstrapping argument to ensure that each function $u_\varepsilon$ is actually smooth.
	\end{proof}

\bigskip

\begin{rem} One may wonder wether we could also retrieve the second part of the result by Biquard and Rollin \cite{BiqRol}, Theorem D -namely, the minimizing property. To do this, one would need to check that the results obtained by Schoen and Wolfson \cite{SchoWolf} can be extended to the almost-Kähler setting. 
\end{rem}

\section*{\textsc{Annex}: ALE metric on \texorpdfstring{$T^*S^2$}{the cotangent of the sphere} as a smoothing of the $A_1$ singularity.}

We recall some results from the last part of Stenzel's paper \cite{Ste}. 

Consider thea singularity $\CZ$ endowed with the Euclidean Kähler structure $(J_0, \omega_0, g_0)$. We identify $\CZ$ to the cone
\begin{equation*}
	\mathcal{C} = \{z \in \C^3, z_1^2 + z_2^2 + z_3^2 = 0\} \subset \C^3.
\end{equation*}
and we consider smoothings of the form
\begin{equation}\label{eq:defQeps}
  C_\varepsilon = \{ z \in \C^3, z_1^2+z_2^2+z_3^2 = \varepsilon^2\},
\end{equation}
endowed with the restriction of the natural complex structure on $\C^3$. 
Here $\varepsilon$ is a positive real number. The construction would actually make sense for a complex parameter $\varepsilon$. In that case, we would retrieve the family of hyperKähler metrics on $O(-2)$ that were obtained by Kronheimer \cite{Kro}. However, this will not intervene in our construction.

We now recall the construction of the Ricci-flat Kähler metric on $C_\epsilon$ obtained by Stenzel in \cite{Ste}. 

We denote by $\tau = |z|^2_{|C_\epsilon}$ the restriction of the squared norm in $\C^3$ to the quadric $C_\epsilon$, and we look for a Kähler potential under the form $u = f \circ \tau$. To find a Ricci-flat metric, we wish to solve the Monge-Ampère equation :
\begin{equation} \label{mongeampere}
  \Ric(\omega_u) = -i\del\delbar \log\det(u_{i\bar j}) = 0,
\end{equation}
where the subscripts denote derivation with respect to local coordinates on $C_\epsilon$.

Using proper coordinates, a straightforward if somewhat tedious computation, which can be found in Patrizio and Wong (\cite{PatWong}), shows that $f\circ \tau$ is a solution of the Monge-Ampère equation \eqref{mongeampere} whenever $f$ satisfies the following ODE~:
\begin{equation}\label{eq:MAode}
  \tau f'(\tau)^2 + f''(\tau)f'(\tau)(\tau^2-\varepsilon^4) = c,
\end{equation}
where $c$ is a positive constant.

\medskip
This EDO, together with sensible initial conditions, admits $f(\tau) = \sqrt{\tau+\varepsilon^2}$ as the unique solution. The Ricci-flat Kähler metric associated to this potential will be denoted $\omega_{X,\varepsilon}$ on $C_\epsilon$. The associated Riemannian metric $g_{X,\varepsilon}$ coincides with a rescaling of the Eguchi-Hanson metric; however, the complex structure $J_{X,\varepsilon}$ differs from the standard one, as explained earlier.\\

To study the ALE character of this metric, observe that $C_\epsilon$ can be identified to $T^*S^2$. Indeed, separating the real and imaginary parts, we have 
\begin{equation*}
	C_\epsilon = \{X+iY, (X,Y) \in \R^3 \times \R^3\ |\ \langle X,X\rangle - \langle Y,Y\rangle = \epsilon,\ \langle X,Y\rangle = 0 \},	
\end{equation*}
whereas
\begin{equation*}
	T^*S^2 = \{(X, \xi) \in \R^3 \times \R^3\ |\ \|X\| = 1, \langle X, \xi \rangle =0 \},
\end{equation*}
thus the map
\begin{align*}
   \Psi_\varepsilon: T^*S^2 &\rightarrow Q_\varepsilon\\
    (x,\xi) &\mapsto (\varepsilon\cosh(\|\xi\|)x,\varepsilon \frac{\sinh(\|\xi\|)}{\|\xi\|}\xi.
\end{align*}
 identifies the smoothing $C_\varepsilon$ with the cotangent of the sphere. 

\medskip

\begin{rem}
	This maps the zero section $S^2 =\{(x, 0), \|x\| = 1\} \subset T^*S^2$ to the subset $\{(\varepsilon x, 0), \|\| =1 \} \subset Q_\varepsilon$. When $\varepsilon$ goes to 0, the section nulle is collapses on the singular point $(0,0) \in \CZ$.
\end{rem}

Using spherical coordinates on  $T^*S^2 \setminus S^2$ outside the zero section, we see that the Ricci-flat Kähler structure we have obtained on $C_\varepsilon$ pulls back to
\begin{align*}
J_{X,\varepsilon} \dfrac{\del}{\del t} &= -X_3,\ J_SX_1 = -\tanh(t)X_2\\
\omega_{X,\varepsilon} &=  \sqrt{2}\varepsilon(\cosh(t)\ \alpha_3 \wedge dt + \sinh(t)\ \alpha_2 \wedge \alpha_1)\\
g_{X,\varepsilon} &= \sqrt{2}\varepsilon(\cosh(t)\ dt^2 + \sinh(t)\tanh(t)\ \alpha_1^2 + \cosh(t)(\alpha_2^2 + \alpha_3^2)).
\end{align*}

To compare to the Euclidean metric, rather than to the conical one, on $\CZ$, we change variables radially, setting $\cosh(t) = \frac{s^2}2$. This gives
\begin{align*}\label{eq:formuleJS}
J_{X,\varepsilon}\dfrac{\del}{\del s} &= -\dfrac{2s}{\sqrt{s^4 -4}}X_3,\ J_S X_1 = -\sqrt{1-\frac4{s^4}}X_2\\
\frac1{\sqrt{2}\varepsilon}g_{X,\varepsilon}&= \left(1-\dfrac4{s^4}\right)^{-1}ds^2 + \dfrac{s^2}4 \left(1-\dfrac4{s^4}\right) \alpha_1^2 + \dfrac{s^2}4(\alpha_2^2 + \alpha_3^2),\\
\omega_{X,\varepsilon} &=  \sqrt{2}\varepsilon \left(\dfrac{s}{2\sqrt{1 - \frac{4}{s^4}}} \alpha_3 \wedge ds + \frac{s^2}4 \sqrt{1- \frac{4}{s^4}}\  \alpha_2\wedge \alpha_1 \right)
\end{align*}

Comparing to the Euclidean structure:
\begin{align*}
J_0\dfrac{\del}{\del s} &= -\dfrac{2}{s}X_3,\ J_0X_1 = -X_2\\
g_0 &= ds^2 + \frac{s^2}4(\alpha_1^2+\alpha^2_2 + \alpha_3^2),
\end{align*}
we see that the derivatives of the coefficient at any order verify
\begin{align*}
  \del^j(J_{X,\varepsilon}-J_0) &= O(s^{-4-j})\\
 \del^j(\sqrt{2}\varepsilon g_{X,\varepsilon}-g_0) &= O(s^{-4-j});
\end{align*}
thus the metric is ALE of order 4.

\begin{rem} 
We recognize a rescaling of the Eguchi-Hanson metric on $T^*S^2$. Howerver, the complex structure is different from the one on $T^*\C P^1 = O(-2)$ obtained when blowing up the origin in $\CZ$. Indeed, instead of an exeptional divisor biholomorphic to $\C P^1$ (corresponding to the zero section), we have a Lagrangian 2-sphere.
\end{rem}

\bibliographystyle{plain}
\bibliography{biblio.bib}

\end{document}